\documentclass[reqno,12pt]{amsart}
\usepackage{amsmath,amsthm,amssymb,amsfonts,epsfig,color,graphicx,enumerate,enumitem,accents,subfigure}
\usepackage[a4paper,margin=3truecm]{geometry}
\usepackage[T1]{fontenc}

\DeclareMathOperator{\spt}{spt}

\def\med{\medskip\noindent}
\def\mutild{\tilde{\mu}}
\def\rhotild{\tilde{\rho}}
\def\a{\alpha}

\def\f{\varphi}

\def\ds{\displaystyle}
\def\eps{{\varepsilon}}
\def\N{\mathbb{N}}
\def\R{\mathbb{R}}

\def\O{\Omega}
\def\A{\mathcal{A}}
\def\B{\mathcal{B}}
\def\C{\mathcal{C}}

\def\M{\mathcal{M}}
\def\PP{\mathcal{P}}

\def\ctild{\tilde{c}}
\def\Ctild{\tilde{C}}
\def\Ptild{\tilde{P}}
\def\rhotild{\tilde{\rho}}

\def\Cbar{\overline{C}}
\newcommand{\be}{\begin{equation}}
\newcommand{\ee}{\end{equation}}
\newcommand{\bib}[4]{\bibitem{#1}{\sc#2: }{\it#3. }{#4.}}

\newcommand{\weak}{\stackrel{*}{\rightharpoonup}}
\newcommand{\norm}[1]{\lvert #1 \rvert}
\newcommand{\Norm}[1]{\lVert #1 \rVert}
\newcommand{\res}{\mathbin{\vrule height 1.6ex depth 0pt width 0.13ex\vrule height 0.13ex depth 0pt width 1.3ex}}

\numberwithin{equation}{section}
\theoremstyle{plain}
\newtheorem{theo}{Theorem}[section]
\newtheorem{lemm}[theo]{Lemma}
\newtheorem{coro}[theo]{Corollary}
\newtheorem{prop}[theo]{Proposition}

\theoremstyle{remark}
\newtheorem{rema}[theo]{\bf Remark}
\newtheorem{exam}[theo]{Example}

\newenvironment{ack}{{\bf Acknowledgements. }}

\title[Relaxed multi-marginal costs and quantization effects]{Relaxed multi-marginal costs and quantization effects}

\author{Guy Bouchitt\'e, Giuseppe Buttazzo, Thierry Champion, Luigi De Pascale}

\date{\today}

\begin{document}

\begin{abstract} We propose a duality theory for multi-marginal repulsive cost that appear in optimal transport problems arising in Density Functional Theory. The related optimization problems involve probabilities on the entire space and, as minimizing sequences may lose mass at infinity, it is natural to expect relaxed solutions which are sub-probabilities. We first characterize the $N$-marginals relaxed cost in terms of a stratification formula which takes into account all $k$ interactions with $k\le N$. We then develop a duality framework involving
continuous functions vanishing at infinity and deduce primal-dual necessary and sufficient optimality conditions   Next we prove the existence and the regularity of an optimal dual potential under very mild assumptions. In the last part of the paper, we apply our results to a minimization problem involving a given continuous potential and we give evidence of a mass quantization effect for optimal solutions.
\end{abstract}

\maketitle



\section{Introduction}\label{sintro}

An interesting issue in {\it Density Functional Theory} (DFT), an important branch of Quantum Chemistry, is to understand the asymptotic behavior as $\eps\to0$ of the infimum problem
\be\label{mineps}
\min\left\{\eps T(\rho)+C(\rho)-U(\rho)\ :\ \rho\in\PP\right\}
\ee
where the parameter $\eps$ stands for the Planck constant and
\begin{itemize}
\item$T(\rho)$ is the kinetic energy
$$T(\rho)=\frac{1}{2}\int_{\R^d}|\nabla\sqrt\rho|^2\,dx;$$
\item$C(\rho)$ describes the electron-electron interaction;
\item$U(\rho)$ is the potential term
$$U(\rho)=\int_{\R^d}V(x)\rho\,dx;$$
\item$\PP$ is the class of all probabilities over $\R^d$.
\end{itemize}
The term $C(\rho)$ is the one on which we focus our attention. Here we want to stress that the ambient space is the whole $\R^d$ ($d=3$ in the physical applications); for simplicity integrals over $\R^d$ are often denoted without the indication of the domain of integration, and similarly for spaces of functions or measures defined over all $\R^d$ we do not indicate the domain. Starting from the works \cite{bdpgg12} and \cite{cfk13} the link between optimal transportation problems \cite{vi03}  and DFT for Coulomb systems \cite{lieb83} has been investigated and in particular the term $C(\rho)$ has been considered as a multi-marginal transport cost:
$$C(\rho)=\inf\left\{\int_{\R^{Nd}}c(x_1\,\dots,x_N)\,dP\ :\ P\in\Pi(\rho)\right\}$$
where
$$c(x_1\,\dots,x_N)=\sum_{1\le i<j\le N}\frac{1}{|x_i-x_j|}$$
and $\Pi(\rho)$ is the family of multi-marginal transport plans
$$\Pi(\rho)=\left\{P\in\PP(\R^{Nd})\ :\ \pi_i^\#P=\rho\hbox{ for all }i=1,\dots,N\right\}$$
being $\pi_i$ the projections from $\R^{Nd}$ on the $i$-th factor $\R^d$ and $\pi_i^\#$ the push-forward operator
$$\pi_i^\#P(E)=P\big(\pi_i^{-1}(E)\big)\qquad\hbox{for all Borel sets }E\subset\R^d.$$
The relationship of the transport cost $C(\rho)$ with the DFT is mainly related to the fact that it is the semiclassical limit of the Levy-Lieb energy as shown in 
\cite{bdp17,cfk13,cfk17,lew17}.

If $U$ is associated with a Coulomb potential $V:\R^d\to[0,+\infty]$ of the kind
$$V(x)=\sum_{k=1}^M\frac{Z_k}{|x-X_k|}\qquad Z_k>0,\ X_k\in\R^d,$$
where for $k=1,\dots,M$ the $X_k$ are the positions of the nuclei and $Z_k$ the corresponding charges, it turns out that the infimum in \eqref{mineps} blows up to $-\infty$ as $\eps^{-1}$. Studying the rescaled problem as $\eps\to0$ is a quite difficult issue due to the presence of the strong interaction term $C(\rho)$. In particular it is not clear under which conditions on the potential $V$ the problem \eqref{mineps} admits a solution and what is the limiting problem as $\eps\to 0$ in the sense of $\Gamma$-convergence. A partial answer is known when the electronic energy $C(\rho)$ involves two electrons only (see \cite{bbcd18}). Some theoretical and numerical results on the same problem are also contained in \cite{cfm14}.

In contrast, if instead of the Coulomb potential we choose a continuous potential $V$ vanishing at infinity, then the infimum in \eqref{mineps} remains bounded as $\eps\to 0$. It is important to notice that limits of minimizing sequences $(\rho_\eps)$ are not in general probabilities since some mass can be lost at infinity (what we call {\it ionization} phenomenon \cite{fnvdb18, fnvdb18b, solo03, solo91}). It turns out by an elementary $\Gamma$-convergence argument that, in this case,
the weak* limits of $(\rho_\eps)$ can be characterized as solutions of the limit problem
 \be\label{minrelax0}
\min\left\{\Cbar (\rho)- \int V\,d\rho\ :\ \rho\in\PP^-\right\}
\ee
where $\Cbar(\rho)$ is the relaxation of the functional $C(\rho)$ defined by:
$$\Cbar(\rho)=\inf\left\{\liminf_n C(\rho_n)\ :\ \rho_n\weak\rho,\ \rho_n\in\PP\right\}.$$
Notice that $\Cbar(\rho)$ is defined for all $\rho$ belonging to the class of sub-probabilities
$$\PP^-=\left\{\rho\hbox{ nonnegative Borel measure on }\R^d\ :\ \|\rho\|\le1\right\}$$
where by $\|\rho\|$ we simply denoted the mass of $\rho$
$$\|\rho\|=\int d\rho.$$

The minimization problem \eqref{minrelax0} is convex and our first goal is to develop a duality theory for the optimal transport problem related to the cost functional $\Cbar(\rho)$. This is achieved by considering the compactification of $\R^d$ through the addition of a point $\omega$ at infinity and the related dual space $C_0\oplus\R$. The duality formula is illustrated in Theorem \ref{duality}:
\be\label{introdual}
\Cbar(\rho)=\sup\left\{\int\psi\,d\rho+(1-\|\rho\|)\psi_\infty\ :\ \psi\in\A\right\}
\ee
where $\A$ is the class of admissible functions, defined as
$$\A=\left\{\psi\in C_0\oplus\R\ :\ \frac1N\sum_{i=1}^N\psi(x_i)\le c(x_1,\dots,x_N)\quad \forall x_i\in(\R^d)^N\right\},$$
and $\psi_\infty$ denotes the limit of $\psi$ at infinity.
Furthermore we prove for a large class of $\rho$ the existence of an optimal Lipschitz continuous potential $\psi$ for \eqref{introdual}. 
We are also able to represent the relaxed cost functional $\Cbar$ through a {\it stratification} formula:
\be\label{strat}
\Cbar(\rho)=\inf\left\{\sum_{k=1}^N\C_k(\rho_k)\ :\ \rho_k\in\PP^-,\ \sum_{k=1}^N\frac{k}{N}\rho_k=\rho,\ \sum_{k=1}^N\|\rho_k\|\le1\right\}
\ee
which makes use of all {\it partial interaction functionals} $\C_k, 1\le k\le N$ as defined in \eqref{Cextend}. We recently become aware that some results on the same spirit but in the framework of Grand Canonical Optimal Transportation are currently being obtained by Di Marino, Lewin and Nenna \cite{dmnele19}.

As an application of our duality theory we analyze the optimization problem \eqref{minrelax0} with a potential $V(x)$ belonging to $C_0$. Even in this simplified case, when no kinetic energy is present, we discover a very rich and surprising structure in which, depending on the choice of $V$, we obtain optimal solutions which are either probabilities or sub-probabilities with fractional mass $\frac{k}{N}$ with $k$ integer (phenomenon that we may interpret as a {\it mass quantization effect}).

Most of the results presented in this paper could be extended with little effort to the more general case of
an interaction cost of the form
$$c(x_1\,\dots,x_N)=\sum_{1\le i<j\le N}\ell(|x_i-x_j|)\ ,$$
where the function $\ell: \R^+ \to (0,+\infty]$ is lower semicontinuous and vanishes at infinity.
Some more technical requirements on the function $\ell$ are needed to extend the existence and Lipschitz regularity result of Section \ref{dualexist}.

\smallskip

The structure of the paper is as follow. In Section \ref{srelax}, we first identify the relaxed functional $\Cbar(\rho)$ on $\PP^-$ in an abstract compact framework (Proposition \ref{prop:ctilde}) and then establish the representation formula \eqref{strat}. In Section \ref{dualitySec}, we develop a complete duality framework extending previous works \cite{bcdp17,dep15,kel84} to the case of sub-probabilities and allowing practical computations in case of finitely supported measures (Proposition \ref{rhopp}). In addition we derive some very useful primal-dual necessary and sufficient optimality conditions (Theorem \ref{prop:cns}). The long Section \ref{dualexist} is devoted to the existence and Lipschitz regularity of an optimal dual potential for sub-probabilities. In Section \ref{relaxedmass}, we apply all previous results to the relaxed problem \eqref{minrelax0} and enlighten the mass quantization effect.

\bigskip
Let us list a few notations that will be used constantly along the paper. 
\begin{itemize}
\item[--]$C_b$ is the space of continuous and bounded functions in $\R^d$  equipped with the $\sup$-norm;  
\item[--] $C_0$ is the separable Banach subspace of $C_b$ consisting of those functions vanishing at $\infty$ and $C_0^+$ is subclass of nonnegative elements of $C_0$;
\item[--] $Lip $ is the space of Lipschitz continuous functions on $\R^d$ and for an element $\varphi \in Lip$ the Lipschitz semi-norm will be denoted by $Lip(\varphi)$. $Lip_k$ is the subset of Lipschitz functions with Lipschitz constant equal to $k$;
\item[--] $\psi_+$ denotes the positive part of a function $\psi$, i.e. $\max\{\psi, 0\}$; 
\item[--] $\PP(E)$ is the set of Borel probability measures on the metric space $E$; 
\item[--] $\PP^-(E)$ is the set of Borel sub-probability measures  on the metric space $E$, i.e. the set of positive measures $\rho$ of total variation $\|\rho\|\leq 1$. $\PP^-$ will be  equipped with the weak* convergence. With this convergence $\PP^-$ is a compact metrizable space. 
\end{itemize}
\section{The relaxed multi-marginal cost}\label{srelax}

A crucial step for the proof of the existence of an optimal $\rho\in\PP$ for the minimization problem \eqref{minrelax0} is the study of the relaxed cost
\[
\Cbar(\rho)=\inf\left\{\liminf_n C(\rho_n)\ :\ \rho_n\weak\rho,\ \rho_n\in\PP\right\}
\]
of the electron-electron interaction functional $C(\rho)$. Note that, while $C(\rho)$ is defined on probabilities $\rho\in\PP$, since the weak* convergence may allow loss of mass at infinity, the relaxed cost $\overline C(\rho)$ is defined for $\rho\in\PP^-$. The complete characterization of $C(\rho)$ is obtained in Subsection \ref{stratmain}. In a first step we derive $\Cbar(\rho)$ in a abstract way by embedding $\R^d$ into its Alexandroff compactification.

\subsection{A compact framework for the relaxed cost}

It is  convenient to study this cost in a compact framework, by embedding the elements $\rho\in\PP^-$ as probabilities in a compact space. To this end, we introduce a point $\omega$ at infinity, and we denote by $X=\R^d\cup\{\omega\}$ the compact set resulting from Alexandrov's construction.
\def\injRd{{S}}
We also denote $\injRd:x \mapsto x$ the identity embedding of $\R^d$ into $X$, and consider the transformed Coulomb cost $\ctild$ on $X^N$ given by
\be\label{ctil}
\ctild(x_1,\ldots,x_N) = \sum_{1 \leq i < j \leq N} \frac{1}{\norm{x_i-x_j}}
\ee
where we set $1/\norm{a-b} = 0$ whenever $a$ or $b$ equals $\omega$.
Note that this convention yields that $\ctild$ is lower semi-continuous on $X$.
We can now define the transport cost $\Ctild$ for any $\rhotild\in\PP(X)$ by
\[
\Ctild(\rhotild) := \min\left\{\int_{X^N}\ctild\,d\Ptild\ :\ \Ptild\in\PP(X^N),\ \Ptild\in\Pi(\rhotild) \right\}\,.
\]
Note that $\Ctild$ is lower semi-continuous on $\PP(X)$ endowed with the weak topology, i.e. the topology of narrow convergence for measures on the compact set $X$. The relation between $C$ and $\Ctild$ is as follows: for $\rho\in\PP$ we have that $\rhotild:=\injRd^\#\rho$ belongs to $\PP(X)$, and
$$P\in\Pi(\rho)\ \Longleftrightarrow\ \Ptild:=\big(\injRd^{\otimes N}\big)^\# P\in\Pi(\rhotild)$$
where we use the notation $f^{\otimes N}=f\otimes\dots\otimes f$ ($N$ times). With these notations we have
$$\int_{(\R^d)N} c\,dP=\int_{X^N}\ctild\,d\Ptild.$$
As a consequence $C(\rho)=\Ctild\big(\injRd^\#\rho\big)$ whenever $\rho\in\PP$. The following result now relates $\Cbar$ and $\Ctild$.

\begin{prop}\label{prop:ctilde}
For every $\rho\in\PP^-$ it holds
\be\label{eq:ctilde}
\Cbar(\rho)=\Ctild(\rhotild)\qquad\text{for}\quad\rhotild:=S^\#\rho+(1-\|\rho\|)\delta_\omega.
\ee
\end{prop}

\begin{proof}
We denote by
$$\Gamma(\rho):=\Ctild\big(S^\#\rho+(1-\|\rho\|)\delta_\omega\big)$$
the right hand side of \eqref{eq:ctilde}. From the preceding discussion we clearly have
$$\Cbar(\rho)=C(\rho)=\Gamma(\rho)\qquad\text{whenever }\rho\in\PP.$$
We first claim that $\Gamma$ is weakly* lower semicontinuous on $\PP^-$. Indeed, assume that $\rho_n \weak \rho$ in $\PP^-$, and consider the probabilities over $X$
$$\rhotild_n=S^\#\rho_n+(1-\|\rho_n\|) \delta_\omega.$$
Then the sequence $(\rhotild_n)_n$ is weakly* compact in $\PP(X)$ so that $\rhotild_n \weak \rhotild$ for some $\rhotild\in\PP(X)$. We then infer $\rhotild\res\R^d=\rho$, so that in fact $\rhotild=S^\#\rho+(1-\|\rho\|)\delta_\omega$ and
\[
\liminf_n \Gamma(\rho_n) \;=\; \liminf_n \Ctild(\rhotild_n)
\;\geq\; \Ctild(\rhotild) \;=\; \Gamma(\rho)\,.
\]
This proves the claim. Since $\Cbar$ is the largest weakly* lower semicontinuous functional on $\PP^-$ which is lower than $C$ on $\PP$, we conclude that $\Cbar \geq \Gamma$.

We now turn to the opposite inequality $\Cbar\le\Gamma$. Let $\rho\in\PP^-$, fix
$\rhotild:=S^\#\rho+(1-\|\rho\|)\delta_\omega$ and $\Ptild\in\Pi(\rhotild)$ a symmetric plan such that
\[
\Gamma(\rho) = \Ctild(\rhotild) = \int_{X^N} \ctild \,d\Ptild\,.
\]
We fix $N$ distinct vectors $\xi_1,\ldots,\xi_N$ on the unit sphere $\R^d$, and for any integer $n$ we define the Borel map $h_n:X^N\to(\R^d)^N$ by
$$h_n(x_1,\ldots,x_N)=\big(h_{n,1}(x_1),\ldots,h_{n,N}(x_N)\big)$$
where $h_{n,i}:X\to\R^d$ is given by
$$h_{n,i}(x)=\begin{cases}
x&\text{if }x\in B(0,n),\\
2n\xi_i&\text{otherwise.}
\end{cases}$$
Note that on $X^N$ it holds
\begin{equation}\label{ineqhn}
c\circ h_n\le\ctild+\frac{N(N-1)}{2\,n}\max_{i\neq j}\left\{1,\frac{1}{|\xi_i-\xi_j|}\right\}=\ctild+O\left(\frac{1}{n}\right)\,.
\end{equation}
We now define $P_n$ as the symmetrization of $(h_n)^\#\Ptild$, that is
$$P_n=\frac{1}{N!}\sum_{\sigma\in\mathcal{S}_N}\sigma^\#\big((h_n)^\#\Ptild\big)$$
We denote by $\rho_n$ the marginal of $P_n$, then $\rho_n\in\PP$ satisfies $\rho_n\res B(0,n)=\rho\res B(0,n)$. As a consequence we get $\rho_n\weak\rho$, so that from \eqref{ineqhn} we have
\begin{align*}
\Cbar(\rho)\le\liminf_n C(\rho_n)
&\le\liminf_n\int_{(\R^d)^N} c\,dP_n\\
&=\liminf_n\int_{X^N} c\circ h_n\,d\Ptild\\
&\le\int_{X^N}\ctild\,d\Ptild=\Gamma(\rho),
\end{align*}
which concludes the proof.
\end{proof}

\subsection{Stratified representation of the relaxed cost}\label{stratmain}

The formula \eqref{eq:ctilde} in Proposition \ref{prop:ctilde}  allows to recover the representation formula obtained in the case $N=2$ in \cite[Proposition 2.5]{bbcd18} and to generalize it to any value $N\ge2$. To this end, we introduce all partial correlation costs $\C_k$ involving interactions $k$ electrons interactions for $2\le k \le N$. They are defined by setting for any $\mu\in\PP^-$:
\be\label{Cextend}
\C_k(\mu):=\inf\left\{\int_{(\R^d)^N} c_k(x_1,\dots,x_k)\,dP(x_1,\dots,x_k)\ :\ \pi_i^\# P=\mu,\ \forall i=1,\dots,k\right\}
\ee
where transport plans $P$ are now non-negative Borel measures on $(R^{d})^k$ with total mass $\|P\|=\|\mu\|=\int d\mu$ and
\be\label{def:ck}
c_k(x_1,\dots,x_k) := \ctild(x_1,\dots,x_k, \omega , \dots,\omega) = \sum_{1 \le i<j\le k} \frac{1}{\norm{x_i-x_j}}\, ,
\ee
being $\ctild$ defined by \eqref{ctil}.
It is also convenient to define $\C_1$ on $\PP^-$ as $\C_1\equiv0$ (meaning that no interaction exists for a single electron). Note that our initial multi-marginal cost $C(\rho)$ agrees with $\C_N(\rho)$ for $\rho\in\PP$. 

We are now in position to state our stratification representation result:

\begin{theo}\label{prop:cpartial}
For every $\rho\in\PP^-$ it holds
\be\label{rhok}
\Cbar(\rho)=\inf\left\{\sum_{k=1}^N\,\C_k(\rho_k)\ :\ \rho_k\in\PP^-,\ \sum_{k=1}^N\frac{k}{N}\rho_k=\rho,\ \sum_{k=1}^N\|\rho_k\|\le1\right\}.
\ee
Moreover the infimum is attained whenever $\Cbar(\rho)<+\infty$.
\end{theo}

\begin{rema}
At this stage, we notice an important connection with the so-called {\it grand canonical} formulation for the infinite multi-marginal problem. Indeed, if we rewrite the sub-probabilities $\rho_k$ in the form $\rho_k=\alpha_k\nu_k$ with $\|\nu_k\|=1$ and $ 0\le\alpha_k\le1$, we obtain
\be\label{canonical}
\Cbar(\rho)=\inf\left\{\sum_{k=1}^N\alpha_k\frac{k}{N}\,C_k(\nu_k)\ :\
\nu_k\in\PP,\ \sum_{k=1}^N\alpha_k\frac{k}{N}\nu_k=\rho,\ \sum_{k=1}^N\alpha_k\le1\right\}.
\ee
In the grand canonical formulation (see for instance \cite{LLS}), the summation with respect to $k$ in \eqref{canonical} runs from $1$ to $+\infty$. Mixed formulations have been used as well by Cotar and Petrache (see \cite{cope19}) and, very recently, by Di Marino, Lewin and Nenna \cite{dmnele19}. In the present paper, we aim to emphasize the connection with the relaxation framework which is crucial for existence and nonexistence issues. 
\end{rema}

\begin{rema}\label{fractionalrho}
From Theorem \ref{prop:cns} below, we deduce that, for $1\le k\le N$, one has 
\be\label{easy}
\Cbar(\rho)\le\C_k\Big(\frac{N\rho}{k}\Big)\qquad\text{whenever }\|\rho\|=\frac{k}{N}
\ee
This is a consequence of \eqref{rhok} when taking $\rho_k=N\rho/k$ and $\rho_j=0$ if $j\not=k$. We conjecture that the inequality in \eqref{easy} is in fact an equality (this would enlighten the fact that configurations involving an integer number of electrons play a special role). 
\end{rema}

\begin{proof}[Proof of Theorem \ref{prop:cpartial}]
Fix $\rho\in\PP^-$ and consider the associated problem
$$(Q_\rho)\qquad\inf\left\{\sum_{k=1}^N\,\C_k(\rho_k)\ :\ \rho_k\in\PP^-,\ \sum_{k=1}^N \frac{k}{N}\rho_k=\rho,\ \sum_{k=1}^N\|\rho_k\|\le1\right\}.$$
We first claim that $\Cbar(\rho) \geq \inf(Q_\rho)$, and assume without loss of generality that $\Cbar(\rho)<+\infty$.
Let $\Ptild \in\PP(X^N)$ be an optimal symmetric plan for $\Ctild(\rhotild)=\Cbar(\rho)$ in the right hand side of \eqref{eq:ctilde}. We set
$$\mutild_k:=\pi_1^\#\left(\Ptild\res(\R^d)^k\times\{\omega\}^{N-k}\right)$$
for any $k$ in $\{1,\ldots,N\}$, with the convention $(\R^d)^N\times\{\omega\}^0=(\R^d)^N$.
By the symmetry of $\Ptild$, we have
\[
\pi_1^\#\left(\Ptild\res\big(\R^d\times X^{N-1}\big)\right)=
\pi_1^\#\left(\Ptild\res\big(\R^d\times(\R^d\cup\left\{\omega\right\})^{N-1}\big)\right)=\sum_{k=1}^N\binom{N-1}{k-1}\mutild_k\,.
\]
Since $\pi_1^\#\Ptild=\rhotild=S^\#\rho+(1-\|\rho\|)\delta_\omega$, we then infer
$$\rho=\sum_{k=1}^N \binom{N-1}{k-1}{\mutild_k}\res\R^d=\sum_{k=1}^N\frac{k}{N}\nu_k$$
where we have set
$$\nu_k:=\binom{N}{k} {\mutild_k}\res\R^d$$
for all $k$. By the symmetry of $\Ptild$, we also have
$$1=\int_{(\R^d\cup\{\omega\})^N}dP\ge\sum_{k=1}^N\binom{N}{k}\int d\mutild_k=\sum_{k=1}^N\|\nu_k\|\,.$$
As a consequence, the measures $\nu_k$ satisfy the constraints of $(Q_\rho)$. Using the symmetry of $\ctild$ and $\Ptild$ and the definition of $c_k$ in \eqref{def:ck}, we obtain
$$\Cbar(\rho)=\sum_{k=2}^N \binom{N}{k} \int_{(\R^d)^k\times\{\omega\}^{N-k}}\ctild\,d\Ptild=\sum_{k=2}^N \int_{(\R^d)^k} c_k\,dP_k$$
where for each $k\ge2$, we indicate by $P_k$ the Borel sub-probability on $(R^d)^k$ 
$$P_k:=\binom{N}{k}{\pi_{1,\ldots,k}}^\#\Ptild$$
being $\pi_{1,\ldots,k}:(\R^d)^N\to(\R^d)^k$ the projection on the $k$ first copies of $\R^d$.
Then for any $k$ the transport plan $P_k$ has marginals $\nu_k$ so that
\[
\Cbar(\rho)=\sum_{k=2}^N \int_{(\R^d)^k} c_k\,dP_k\ge\sum_{k=2}^N \C_k(\nu_k) \geq \inf(Q_\rho)
\]
which proves the claim. Note that, under the hypothesis $\Cbar(\rho)<+\infty$, the equality $\Cbar(\rho)=\inf(Q_\rho)$ would directly yield that the family $\nu_2,\ldots,\nu_N$ is a solution of $(Q_\rho)$.

We now prove the reverse inequality $\Cbar(\rho)\le\inf(Q_\rho)$, and assume without loss of generality that $\inf(Q_\rho)<+\infty$. We consider $\rho_1,\ldots,\rho_N$ admissible for $(Q_\rho)$ such that $\sum_{k=1}^N \C_k(\rho_k)<+\infty$. For each $k\ge2$ take $P_k\in\PP^-\big((\R^d)^k\big)$ symmetric and optimal for $\C_k(\rho_k)$, we also set $P_1 = \rho_1$ and define for $k \geq 1$ the plans
\[
\Ptild_k:=\left(\left(\injRd^{\otimes k}\right)^\# P_k \right)\otimes \overbrace{\delta_{\omega}\otimes\ldots\otimes\delta_{\omega}}^{N-k \; times}\,.
\]
We now symmetrize the plans $\Ptild_k$ in the following way : for any $k \in \{1,\ldots,N\}$ we define
$$Sym(\Ptild_k):=\binom{N}{k}^{-1}\sum_{I\subset\{1,\dots,N\},\ |I|=k}(\sigma_I)^\#\Ptild_k$$
where $\sigma_I$ is the permutation of $\{1,\ldots,N\}$ which is increasing on $\{1,\ldots,k\}$ with image $I$ and increasing on $\{k+1,\ldots,N\}$. By convention if $x \in X^N$ we set $\sigma_I(x) = (x_{\sigma(1)},\ldots,x_{\sigma(N)})$.
Finally we define
$$\Ptild^*:=\sum_{k=1}^N Sym(\Ptild_k)$$
and we note that $\Ptild^*$ is a sub-probability on $X^N$ since
$$\int_{X^N}\,d\Ptild^*=\sum_{k=1}^N\int_{X^N}\,d\Ptild_k=\sum_{k=1}^N\Norm{P_k}=\sum_{k=1}^N \Norm{\rho_k}\le1$$
where the last inequality follows from the constraint in $(Q_\rho)$. We can then define on $X^N$ the probability
\[
\Ptild = \Ptild^* + (1-\Norm{\Ptild^*})\delta_{\omega} \otimes \ldots \otimes \delta_{\omega}\,.
\]
We now compute the first marginal $\rhotild={\pi_1}^\#\Ptild$: since it is a probability over $X$ it is sufficient to consider its restriction to $\R^d$, which gives
\[
\rhotild\res\R^d
=\sum_{k=1}^N\binom{N}{k}^{-1}\!\!\!\!\!\!\sum_{I\subset\{1,\dots,N\},\ |I|=k}{\pi_1}^\#\left((\sigma_I)^\#\Ptild_k \right)\res \R^d
= \sum_{k=1}^N \binom{N}{k}^{-1} \binom{N-1}{k-1} \rho_k=\rho
\] 
where we used the fact that
$${\pi_1}^\#\left((\sigma_I)^\#\Ptild_k\right)\res\R^d=0\quad\text{whenever}\quad1\notin I.$$
As a consequence $\rhotild=\injRd^\#\rho+(1-\Norm{\rho})\delta_\omega$. We now infer from \eqref{eq:ctilde} that
$$\Cbar(\rho)=\Ctild(\rhotild)\le\int_{X^N} \ctild d\Ptild
=\sum_{k=1}^N\int_{X^N}\ctild\,d\Ptild_k=\sum_{k=1}^N\C_k(\rho_k)$$
which concludes the proof.
\end{proof}

We conclude this Section by a monotonicity formula for the partial interaction costs $\C_k$.

\begin{prop}
Let $\mu\in\PP^-$, then it holds
$$\forall k\ge l,\qquad\C_k(\mu)\ge\frac{k(k-1)}{l(l-1)}\,\C_l(\mu).$$
In particular, one has
$$\forall k\ge1,\qquad\C_{k+1}(\mu)\ge\frac{k+1}{k-1}\,\C_k(\mu).$$
\end{prop}

\begin{proof}
Without loss of generality we assume $\C_k(\mu)<+\infty$, and we denote by $P_k$ a symmetric measure in $\Pi_k(\mu)$ such that
$$\C_k(\mu)=\int c_k\,dP_{k}.$$
We define the measures $P_2:=\pi_{1,2}^\# P_k$ and $P_l:=\pi_{1,\ldots,l}^\# P_k$ to be the push-forward of $P_k$ respectively by the projection on the $2$ and $l$ first spaces $\R^d$. We note that these two measures have $\mu$ as marginals and that
$$\pi_{1,2}^\# P_l=\pi_{1,2}^\# P_k=P_2.$$
From the symmetry of $P_k$ and $P_l$ we can compute
$$\C_k(\mu)=\int c_k\,dP_k=\binom{k}{2}\int c_2\,dP_2
\quad\text{and}\quad
\binom{l}{2}\int c_2\,dP_2=\int c_l\,dP_l\ge\C_l(\mu)$$
from which the inequality follows.
\end{proof}

\section{Dual formulation of the relaxed cost}\label{dualitySec}

This Section is devoted to a duality formula for $\Cbar(\rho)$. We consider the separable Banach space $C_0\oplus\R$ consisting of all continuous functions $\psi:\R^d\to\R$ with a constant value at infinity, that is of the form $\psi=\f+\kappa$ with $\f\in C_0$ and $\kappa\in\R$. Then we consider the closed convex subset
\begin{equation}\label{defA}
\A=\left\{\psi\in C_0\oplus\R\ :\ \frac1N\sum_{i=1}^N \psi(x_i) \le c(x)\quad \forall x \in (\R^d)^N \right\}
\end{equation}
It is convenient to introduce also a larger convex set namely
\be\label{defB}
\B=\left\{\psi\in \mathcal{S}\ :\ \frac1N\sum_{i=1}^N \psi(x_i) \le c(x)\quad \forall x \in(\R^d)^N\right\}
\ee
where $\mathcal{S}$ denotes the set of lower semicontinuous functions $\psi:\R^d\to\R$ such that $\inf \psi>-\infty$. To any such a function $\psi$, we associate the real numbers
$$\psi_\infty:=\liminf_{|x|\to+\infty}\psi(x)=\lim_{R\to+\infty}\Big(\inf\{\psi(x)\ :\ |x|\ge R\}\Big).$$
$$\psi^\infty:=\limsup_{|x|\to+\infty}\psi(x)=\lim_{R\to+\infty}\Big(\sup\{\psi(x)\ :\ |x|\ge R\}\Big).$$
By induction on the integer $N$, it is easy to check that 
$$\psi^\infty\le0\qquad\text{for every }\psi\in\B.$$
A particular choice of such a function in $\B$ is provided in Example \ref{exam1} hereafter. We will use the following truncation lemma.

\begin{lemm} \label{troncation}
Let $\psi$ belongs to $\A$ (resp. to $\B$) and let $\lambda\le \psi^\infty$. Then the function $\psi_\lambda:= \max\{\psi, \lambda\}$ also belongs to $\A$ (resp. to $\B$).
\end{lemm}

\begin{proof}
We have only to check that $\psi_\lambda$ still satisfies the inequality constraint appearing in the definitions of $\A$ and $\B$. Let $(x_1,\ldots,x_N) \in (\R^d)^N$ and set $I=\{i : \psi_\lambda(x_i)=\lambda\}$. Then consider sequences $(y_i^n)_n$ such that
$$|y_i^n|\to+\infty,\qquad\lim_{n\to\infty}\psi(y_i^n)\ge\lambda,\qquad|y_i^n-y_j^n|\to+\infty\ \text{whenever }i\neq j.$$
Then we have
\begin{align*}
\frac1N\sum_{i=1}^N \psi_\lambda(x_i) & \le \lim_{n \to +\infty} \frac1N \left(\sum_{i \in I} \psi(y_i^n) + \sum_{j \notin I} \psi(x_j) \right) \\
&\le\lim_{n\to+\infty}c\big((y_i^n)_{i\in I}, (x_j)_{j\notin I}\big)\le c(x)
\end{align*}
where, in the second inequality, we used the fact that all the terms $1/|y_i^n-y_j^n|$ and $1/|x_j-y_i^n|$ vanish as $n\to\infty$.
\end{proof} 

An important issue is the following duality representation of $\Cbar(\rho)$ which extends to the case $\|\rho\|<1$ the formula obtained in \cite{bcdp17} for $\|\rho\|=1$. 

\begin{theo}\label{duality}
For every $\rho\in\PP^-$, the following equalities hold
\be\label{cdual}
\Cbar(\rho)=\sup_{\psi\in\A}\left\{\int\psi\,d\rho+(1-\|\rho\|)\psi_\infty\right\}=\sup_{\psi \in\B}\left\{\int\psi\,d\rho+(1-\|\rho\|)\psi_\infty\right\},
\ee
where the classes $\A$ and $\B$ are defined in \eqref{defA} and \eqref{defB}.
\end{theo}

\begin{rema}\label{u>uinfty}
By Lemma \ref{troncation}, the two equalities in \eqref{cdual} are still valid if we restrict the supremum to those functions $\psi$ such that $\psi\ge\psi_\infty$. This can be easily checked by substituting an admissible $\psi$ by the function $\max\{\psi,\psi_\infty\}$ which is still admissible 
with a larger energy. Note that such a function it holds $\psi_\infty=\psi^\infty$. 
\end{rema}

\begin{coro} \label{smallermass}
Let $\rho_1,\rho_2$ in $\PP^-$ such that $\rho_1\le\rho_2$. Then $\Cbar(\rho_1)\le \Cbar(\rho_2)$.
\end{coro}

\begin{proof}\ Let us rewrite \eqref{cdual} as
$$\Cbar(\rho)=\sup_{\psi\in\A}\left\{\int(\psi-\psi_\infty)\,d\rho+\psi_\infty\right\}.$$
In view of Remark \ref{u>uinfty}, we may assume that $\psi-\psi_\infty\ge0$ from which the desired inequality follows.
\end{proof}

\begin{rema}\label{compactpart}
In view of the compactification procedure introduced in Section \ref{srelax}, we may extend any function $\psi\in\mathcal{S}$ to $X=\R^d\cup\{\omega\}$ by setting $u=\psi$ on $\R^d$ and $u(\omega)=\psi_\infty$. Notice that, by construction, $u$ is lower semicontinuous as a function on $X$ (i.e. $u\in \mathcal{S}(X)$) and that it is continuous if and only $\psi_\infty=\lim_{|x|\to\infty}\psi(x)$ that is to say $\psi\in C_0\oplus\R^d$. Furthermore, the point-wise constraint for $\psi\in\B$ is equivalent in term of $u$ to the following
\be\label{contrainte}
\frac1N\sum_{i=1}^{i=N}u(x_i)\le\ctild(x_1,x_2,\dots,x_N)\qquad\forall x\in X^N,
\ee
being the extended cost $\ctild$ defined by \eqref{ctil}. Accordingly, if $\rho\in\PP^-$, the representation formula \eqref{cdual} can be rewritten as
$$\Cbar(\rho)= \sup \left\{ \int_X u\,d\tilde{\rho}\ :\ \text{$u\in \mathcal{S}(X)$ satisfies \eqref{contrainte}}\right\},$$ 
where the supremum is taken alternatively in $C(X)$ or in $\mathcal{S}(X)$ and $\tilde{\rho}:=\rho+(1-\|\rho\|)\delta_\omega$ denotes the probability measure on $X$ defined by
$$\int_X u\,d\tilde{\rho}:=\int u\,d\rho+(1-\|\rho\|)u(\omega)=\int u\,d\rho+(1-\|\rho\|)u_\infty.$$
Let us finally notice following equivalence for a sequence $(\rho_n)$in $\PP$ and $\rho\in\PP^-$:
$$\rho_n\weak\rho\quad\Longleftrightarrow\quad\rho_n\to\tilde{\rho}\quad\text{tightly on $X$}.$$
\end{rema}

\begin{proof}[Proof of Theorem \ref{duality}]
For every pair $(\rho,\alpha) \in\M_b\times\R$, we set
$$H(\rho,\alpha):=\begin{cases}
\ds\alpha\,C(\rho/\alpha)
&\text{if $\rho\ge 0$, $\alpha\in\R^+$ and $\|\rho\|=\alpha$\,,}\\
+\infty
&\text{otherwise.} 
\end{cases}$$
As $C$ is convex proper and nonnegative on probability measures, it is easy to check that $H$ is still convex proper nonnegative. In addition it is positively one homogeneous. Therefore the lower semicontinuous envelope of $H$ on $\M_b\times\R$ endowed with its weak star topology can be characterized as the bipolar of $H$ with respect to the duality between $\M_b\times \R$ and $C_0\times\R$, namely
\be\label{dualH}
\begin{split}
\overline{H}(\rho,\alpha)&=\sup\left\{\int\f\,d\rho+\alpha\beta-H^*(\f,\beta)\right\}\\
&=\sup\left\{\int\f\,d\rho+\alpha\beta\ :\ H^*(\f,\beta)\le0\right\}
\end{split}\ee
where the supremum is taken over pairs $(\f,\beta)\in C^0\times\R$ and where in the second equality we exploit the homogeneity of $H$. By the definition of $H$, we infer that:
$$H^*(\f,\beta)\le0\quad\Longleftrightarrow\quad\int\f\,d\rho+\beta\le C(\rho)\quad\forall\rho\in\PP.$$
By the definition of $C(\rho)$, the later inequality is equivalent to
$$\frac1N\int_{R^{Nd}}\sum_{i=1}^N\f(x_i)\,dP+\beta\le\int_{R^{Nd}}c(x)\,P(dx)\qquad\forall P\in\PP.$$
By taking for $P$ a Dirac mass, we may conclude that
$$H^*(\f,\beta)\le0\quad\Longleftrightarrow\quad\frac1N\sum_{i=1}^N\f(x_i)+\beta\le c(x)\qquad\forall x\in\R^{Nd}.$$
Therefore, setting $\psi=\f+\beta$ (thus $\psi_\infty=\beta$) and $\alpha=1$, we deduce from \eqref{defA} and \eqref{dualH} that
$$\overline{H}(\rho,1)=\sup_{\psi\in\A}\left\{\int\psi\,d\rho+(1-\|\rho\|)\psi_\infty\right\}.$$
Therefore to establish the first equality in \eqref{cdual}, we are reduced to show that:
\be\label{CLAIM1}
\Cbar(\rho)=\overline{H}(\rho,1),\qquad\forall\rho\in\PP^-
\ee
The lower bound inequality for $\Cbar(\rho)$ is straightforward since, for every sequence $\rho_n\weak\rho$, it holds
$$\liminf_n C(\rho_n) = \liminf_n H(\rho_n,1)\ge\overline{H}(\rho,1) .$$
To show that $\Cbar(\rho) \le \overline{H}(\rho,1)$ for every $\rho\in\PP^-$, we choose a particular sequence $(\rho_n,\alpha_n)$ in $\M^+\times\R^+$ such that
$$\rho_n\weak\rho,\qquad\alpha_n\to1,\qquad H(\rho_n,\alpha_n)=\alpha_nC\Big(\frac{\rho_n}{\alpha_n}\Big)\to\overline{H}(\rho,1).$$
Then, setting $\rhotild_n:=\rho_n/\alpha_n$, we obtain a sequence of probability measures $(\rhotild_n)$ such that $\rhotild_n\weak \rho$ and $C(\rhotild_n) \to \overline{H}(\rho,1)$. Thus \eqref{CLAIM1} is proved.

In order to prove the second equality in \eqref{cdual}, since the subset $\B$ is larger than $\A$, it is enough
to show that
\be\label{CLAIM2}
\Cbar(\rho)\ge\int\psi\,d\rho+(1-\|\rho\|)\psi_\infty,\qquad\forall\psi\in\B,\ \forall\rho\in\PP^-.
\ee
By \eqref{CLAIM1}, we know that the inequality above holds whenever $\psi$ belongs to $\A$. 
To extend it to $\psi\in \B$, we follow Remark \ref{compactpart} considering the element of $\mathcal{S}(X)$ defined by $u= \psi$ on $\R^d$ and $u(\omega)=\psi_\infty$. As $X$ is a compact metrizable space, we can find a sequence $(u_n)$ in $C^0(X)$ such that
$$u_{n+1}\ge u_n,\qquad\sup_n u_n= u.$$
Clearly the restriction $\psi_n=u_n\res\R^d$ satisfies $\psi_n\le\psi$, thus belongs to $\A$. By applying Beppo Levi's on $X$ equipped with the probability measure $\tilde{\rho}=\rho+(1-\|\rho\|)\delta_\omega$, we obtain:
\[\begin{split}
\lim_n\int\psi_n\,d\rho+(1-\|\rho\|)(\psi_n)_\infty
&=\lim_n\int_X u_n\,d{\tilde\rho}\\
&=\int_X u\,d{\tilde\rho}=\int\psi\,d\rho+(1-\|\rho\|)\psi_\infty\;,
\end{split}\]
from which \eqref{CLAIM2} follows. The proof of Theorem \ref{duality} is then achieved.
\end{proof}

\begin{exam}\label{exam1}
Take for every $R>0$
$$\psi_R(x)=\begin{cases}
(N-1)/(4R)&\text{if }|x|<R,\\
-1/(4R)&\text{if }|x|\ge R.
\end{cases}$$
It is easy to see that $\psi_R\in\B$; indeed, if $x_1,\dots,x_k$ are in the ball $B_R(0)$ and $x_{k+1},\dots,x_N$ are in $\R^d\setminus B_R(0)$, we have to verify that
$$\frac1N\left(k\frac{N-1}{4R}-(N-k)\frac{1}{4R}\right)\le\sum_{1\le i<j\le N}\frac{1}{|x_i-x_j|}\;.$$
Now, the left-hand side above reduces to $(k-1)/(4R)$ while for the right-hand side we have
$$\sum_{1\le i<j\le N}\frac{1}{|x_i-x_j|}\ge\sum_{1\le i<j\le k}\frac{1}{|x_i-x_j|}\ge\frac{k(k-1)}{2}\frac{1}{2R}\;.$$
\end{exam}

As a direct consequence of Theorem \ref{duality}, we obtain

\begin{prop}\label{cbzero}
It holds $\Cbar(\rho)=0$ if and only if $\|\rho\|\le1/N$.
\end{prop}

\begin{proof} Assume first that $\rho$ satisfies $\|\rho\|\le1/N$ and let $\psi\in\A$. By fixing $x_1=x$ and letting $x_2,x_3,\dots x_N$ tend to infinity in different directions in the inequality
$$\frac1{N}\sum_{i=1}^{i=N}\psi(x_i)\le c(x),$$
we infer that
$$\psi(x)+(N-1)\psi_\infty\le0\qquad\forall x\in\R^d.$$
In particular, by sending $|x|$ to infinity, we deduce that $\psi_\infty\le0$. Therefore
$$\int\psi\,d\rho+(1-\|\rho\|)\psi_\infty\le\psi_\infty(1-N\|\rho\|)\le0\qquad\forall\psi\in\A.$$
By \eqref{cdual}, we are led to $\Cbar(\rho)\le0$, thus $\Cbar(\rho)=0$ whenever $\|\rho\|\le1/N$.
 
Let us prove now the converse implication and take an element $\rho\in\PP^-$ such that $\Cbar(\rho)=0$. By \eqref{cdual} for every $\psi\in\B$ we have
$$\int\psi\,d\rho+(1-\|\rho\|)\psi_\infty\le0.$$
In particular, taking as $\psi$ the function $\psi_R$ of Example \ref{exam1}, we have
$$\frac{N-1}{4R}\rho(B_R)-\frac{1}{4R}\rho(B_R^c)-(1-\|\rho\|)\frac{1}{4R}\le0,$$ 
so that
$$(N-1)\rho(B_R)\le\rho(B_R^c)+1-\|\rho\|.$$
Letting $R\to+\infty$ gives
$$(N-1)\|\rho\|\le1-\|\rho\|$$
from which $\|\rho\|\le1/N$.
\end{proof}

\subsection{A weak formulation for dual potentials }\label{weakdual}

The initial motivation of this subsection is to achieve the computation of $\Cbar(\rho)$ through the formula \eqref{cdual} when $\rho$ has a finite support that is of the kind $\rho=\sum_1^K \alpha_i\delta_{a_i}$ where the $a_i\in\R^d$ are distinct, $\alpha_i\ge0$ and $\sum\alpha_i\le 1$. As in Example \ref{exam2} below, we wish to reduce the computation of the supremum in \eqref{cdual} to solving a finite dimensional linear programming problem where the unknown vector involved $y\in\R^{K+1}$ is defined by $y_i=\psi(a_i)$ for $1\le i\le K$ and $y_{K+1}=\psi_\infty$. The linear constraints on the components $y_i$ are deduced simply from
the overall inequalities in $\A$ (or $\B$) by restricting them to the support of $\rho^{N\otimes}$.

\medskip
Then the following issue arises naturally: can we conversely pass from an inequality holding $\rho^{N\otimes}$ almost everywhere to the overall inequality as required in Theorem \ref{duality}? Following the notations introduced in Remark \ref{compactpart}, we can answer this question through the following weak formulation of the dual problem.

\begin{prop}\label{rhopp}
Let $\rho\in\PP^-(\R^d)$ and let $\tilde{\rho}\in \PP(X)$ defined by $\tilde{\rho} =\rho+(1-\|\rho\|)\delta_\omega$. Then
\be\label{dual-pp}
\Cbar(\rho)=\sup\left\{\int_X u\,d\tilde{\rho}\ :\ \frac1N\sum_{i=1}^N u(x_i)\le c(x)\quad\rhotild^{N\otimes}\text{ a.e. $x\in X^N$}\right\},
\ee
being the supremum taken on $\mathcal{S}(X)$ or on $C(X)$.
\end{prop}

\begin{proof} As the admissible set in the right hand side of \eqref{dual-pp} is larger than the one given by \eqref{contrainte}, we have only to prove that, for every $\rho\in\PP^-$, it holds:
\be\label{Claim10}
\Cbar(\rho)\ge\int_X u\,d\tilde{\rho}\qquad\text{for }u\in\tilde{\B}(X),
\ee
where $\tilde{\B}(X)$ denotes the set of elements $u\in\mathcal{S}(X)$ such that
\be\label{constraintpp}
\frac1N\sum_{i=1}^N u(x_i)\le c(x)\qquad\rhotild^{N\otimes}\text{ a.e. }x\in X^N.
\ee 
In a first step, we assume that:
\be\label{Hstep1}
u\in C_0\oplus\R\quad\text{with}\quad u(x)\ge u_\infty.
\ee
First we notice that the inequality in \eqref{constraintpp} holds in fact point-wise in $(\spt(\rho))^N$. Indeed, if $x \in (\spt(\rho))^N$ is such that $c(x)<+\infty$,
then we may integrate the inequality \eqref{constraintpp} on $\Pi_{i=1}^N B(x_i, r)$ and then, dividing by $\Pi_{i=1}^N \, \rho(B(x_i, r))$ and sending $r\to 0$,
we deduce from the continuity of $u$ and $c$ at $x$ that
$$\frac1N\sum_{i=1}^{i=N} u(x_i) \le c(x).$$
Take $\eps>0$. By the lower semicontinuity of $c(x)-\frac1N\sum_i u(x_i)$, the subset
$$\left\{x\in X^N\ :\ \sum_{i=1}^{i=N}u(x_i)<c(x)+\eps\right\}$$
is an open neighborhood of $(\spt(\rho))^N$. Therefore we may chose an open subset $\O_\eps\subset\R^d$ such that:
$$(\spt(\rho))^N\subset\O_\eps\;,\qquad\frac1{N}\sum_{i=1}^{i=N} u(x_i)<c(x)+\eps\text{ for all }x\in(\O_\eps)^N.$$
Let us now define:
$$u_\eps(z):=\begin{cases}
u(z)&\text{if }z\in\O_\eps\\
u_\infty&\text{if }z\in X\setminus\O_\eps.
\end{cases}$$
Then by \eqref{Hstep1}, $u_\eps$ belongs to $\title{S}(X)$. Furthermore it satisfies the overall inequality deduced from \eqref{contrainte} replacing $c$ by $c+\eps$. We are now in position to prove \eqref{Claim10}: choose a sequence $(P_n)$ in $\PP(R^{Nd})$ such that $\Pi(P_n)=\rho_n\weak\rho$ and
$$\Cbar(\rho)=\lim_n C(\rho_n)=\lim_n\int c(x)\,P_n(dx).$$
We obtain
\[\begin{split}
\Cbar(\rho) =\lim_n \int_{X^N} c(x)\,P_n(dx)
&\ge\liminf_n \int_{X^N}\sum_{i=1}^N \frac{u_\eps(x_i)}{N}\,P_n(dx) - \eps\\
&\ge\liminf_n \int_X u_\eps\,d\rho_n-\eps\\
&\ge\int_X u_\eps\,d\rhotild\ -\eps\ ,
\end{split}\]
where in the last line we exploit the lower semicontinuity of $u_\eps$ and the (tight) convergence $\rho_n\to\rhotild=\rho+(1-\|\rho\|)\delta_\omega$. We conclude the proof of \eqref{Claim10} by noticing that $u_\eps=u$ $\rhotild$ a.e. ($u_\eps=u$ on $\spt(\rho)\cup\{\omega\}$).

\med
In a second step, we remove the assumption that $u(x)\ge u_\infty$. Assume first that $\|\rho\|<1$, then $\rhotild$ has a positive mass on $\omega$ and condition \eqref{constraintpp} implies then that, for every $k\in\{1,\dots,N\}$:
$$\frac1N\left(\sum_{i=1}^k u(x_i)+(N-k) u_\infty\right)\le c_k(x_1,x_2,\dots,x_k)\qquad\rhotild^{k\otimes}\text{ a.e. $((x_1,\dots,x_k)\in X^k$}$$
Since $c_k(x_1,x_2,\dots,x_k)\le c(x) $ for every $x\in X$, by setting $v= \sup\{u,u_\infty\}$, we obtain a new continuous function which still satifies \eqref{constraintpp} and such that
$$\int_X v\,d\rhotild\ge\int_X u\,d\rhotild.$$
It is then enough to apply the first step to $v$. If $\|\rho\|=1$, we simply apply the construction of step 1 changing $u_\eps$ into
$$u_\eps(z):=\begin{cases}
u(z)&\text{if }z\in\O_\eps\\
\inf_X u&\text{if }z\in X\setminus{\O_\eps}
\end{cases}.$$
As now $\rhotild$ has no mass on $\omega$, we still have that $u_\eps=u$ $\rhotild$ a.e.
 
\med
Eventually, we drop the continuity assumption by approaching a lower semicontinuous function $u\in \tilde{\B}(X)$ by a sequence of continuous functions $(u_n)$ on $X$ such that:
$$u_{n+1}\ge u_n,\qquad\sup_n u_n=u.$$
Clearly each $u_n$ satisfies the constraint \eqref{constraintpp} so that
$$\Cbar(\rho)\ge\int_X u_n\,d\tilde{\rho}.$$
The conclusion follows by Beppo-Levi's (monotone convergence) Theorem.
\end{proof}

\begin{exam}\label{exam2}
Let $a_1, a_2, a_3\in\R^3$. Our aim is to compute
$$\Cbar(\a_1\delta_{a_1}+\a_2\delta_{a_2}+\a_3\delta_{a_3}):=f(\a_1,\a_2,\a_3)$$
as a function defined on the simplex
$$Q:=\Big\{\a\in\R^3\ :\ \a_i\ge0,\ \sum_i \a_i\le1\Big\}.$$
In order to lighten the calculations, we assume that
$$|a_1-a_2|=|a_2-a_3|=|a_2-a_3|=1$$
and we restrict ourselves to the case $N=3$, where the cost reads
$$c(x)=\sum_{1\le i<j\le 3} \frac1{|x_i-x_j|}\; = \; \frac{1}{|x_1-x_2|}+\frac{1}{|x_1-x_3|}+\frac{1}{|x_2-x_3|}\;.$$
Owing to the representation formula \eqref{dual-pp}, we obtain:
\begin{align*}
f(\a_1,\a_2,\a_3)=\sup\left\{\begin{array}{lll}\ds\sum_{i=1}^3 \a_i\,y_i+(1-\sum_j\a_j)\,y_4\ : &\ds\frac{y_1+y_2+y_3}{3}\le 3\\
y_k +2 y_4\le 0,\quad 1\le k\le 3,&\ds\frac{y_k+ y_l+y_4}{3}\le 1,\quad 1\le k\!<\!l\!\le \!3
\end{array}\right\}\end{align*}
where $y_i$ stands for the value of $u(a_i)$ for $i\in\{1,2,3\}$ while $y_4= u(\omega)$. Rewritten in terms of the nonnegative unknowns
$ x_4= y_4$ and $x_i= 2 x_4- y_i$ for $i\in\{1,2,3\}$, we are led to a classic linear programming:
$$f(\a_1,\a_2,\a_3)=\sup\left\{\Big(3\sum_j \a_j -1\Big) x_4 - \sum_{i=1}^3 \a_i\, x_i\ :\ x\ge 0,\ Ax\le b\right\},$$
being
$$A=\begin{pmatrix} 0&-1& -1& 3\\ -1&0& -1& 3 \\-1& -1&0 &3\\
-1& -1&-1 & 6
\end{pmatrix} \quad,\quad b= \begin{pmatrix} 3\\3\\3\\9\end{pmatrix}$$
It turns out that, for $\a\in[0,\frac13]^3$, only three vertices are involved in the feasible set, namely $(0,0,0,0)$, $(0,0,0,1)$ and $(3,3,3,3)$. We find
$$f(\a)=\begin{cases}
\gamma\left(\sum_{j=1}^3\a_j\right)&\text{if }\a\in[0,\frac1{3}]^3\\
+\infty&\text{otherwise,}
\end{cases}
\qquad\text{with }\gamma(s):=\begin{cases}
0&\text{if }s\le\frac13\\
3s-1&\text{if }\frac13\le s\le\frac23\\
3(2s-1)&\text{if }\frac23\le s\le1.
\end{cases}$$
Notice that here the function $\Cbar(\rho)$, as a function of $\|\rho\|$, is  not differentiable
at $\|\rho\|\in \{\frac13, \frac23\}$. This seems to be a general fact when considering  measures $\rho$
 supported by a set of $M$ points, more precisely if $\|\rho\|=1$, we expect  the function $t\mapsto \Cbar(t\rho)$ to be non-differentiable for fractional masses $t = \frac{k}{M}$.
\end{exam}

\subsection{Optimality primal-dual conditions }\label{Seccns}

By exploiting Theorem \ref{duality} and Theorem \ref{prop:cpartial} (in particular \eqref{rhok} and \eqref{cdual}), we can deduce necessary and sufficient conditions for optimality. It is convenient to introduce, for every $k\in\{1,2,\dots,N\}$ and $\f\in C_0$:
\be\label{def:Mk}
M_k(\f)=\sup\left\{\frac1k\sum_{i=1}^k\f(x_i)-c_k(x_1,\dots,x_k)\right\}
\ee

\begin{lemm}\label{prop:Mk} 
The following properties hold:
\begin{itemize}
\item[i)]The functional $M_k(\f)$ is convex and $1$-Lipschitz on $C_0$. Moreover
\be\label{recession} 
\lim_{t\to+\infty}\frac{M_k(t\f)}{t}=M_1(\f)=\sup\f\;.
\ee
\item[ii)]For every $\f\in C_0$ and $N\in\N^*$, we have:
\be\label{ineqMk}
M_1(\frac{\f}{N})\le\dots\le M_k\Big(\frac{k\f}{N}\Big)\le M_{k+1}\Big(\frac{(k+1)\f}{N}\Big)\le\dots\le M_N(\f).
\ee
\item[iii)]For every $k\in\N^*$ and $\psi\in C_0$, it holds
\be\label{plus}
M_k(\psi)=M_k(\psi_+)\ .
\ee
\end{itemize}
\end{lemm}

\begin{proof}Let us start to prove {\it i)}. The convexity property is straightforward since $M_k$ is a supremum of affine continuous functions. On the other hand, for every $\f_1,\f_2$ in
$C_0$, we obviously have:
$$M_k(\f_2)\le M_k(\f_1)+\sup(\f_2-\f_1)\le M_k(\f_1)+\|\f_2-\f_1\|.$$
Let us now identify the recession function of $M_k$ that is
$$M_k^\infty(\f):= \lim_{t\to +\infty} \, \frac{M_k(t \f)}{t}\;.$$
As $M_k(\f)\le \sup \f$, we clearly have $M_k^\infty(\f)\le \sup \f$. On the other hand, for every $x=(x_i)\in(\R^d)^k$ and $t>0$ it holds
$$\frac{M_k(t\f)}{t}\ge\frac1k\sum_{i=1}^k\f(x_i)-\frac1t c_k(x)\;,$$
so that, after sending $t\to+\infty$ and then optimizing with respect to $x$, we get the converse inequality thus \eqref{recession}.

\med
We prove now {\it ii)}. Let $k\in\{1,2,\dots,N-1\}$ and $\f\in C_0$. Then, for every $x=(x_1,x_2,\dots,x_k,x_{k+1})\in(\R^d)^{k+1}$, it holds:
\begin{align*}M_{k+1}\Big(\frac{(k+1)\f}{N}\Big)
&\ge\frac1N\left(\sum_{i=1}^k \f(x_i)+\f(x_{k+1})\right)-c_{k+1}(x)\\
&\ge\frac1k\left(\sum_{i=1}^k\frac{k\f(x_i)}{N}\right)-c_k(x_1,x_2,\dots,x_k)\;,
\end{align*}
where in the first line we use the definition \eqref{def:Mk}, while in the second line we send $x_{k+1}$ to infinity taking into account that $\f(\omega)=0$.
Finally, optimizing with respect to $x_1,x_2,\dots,x_k$ gives the desired inequality \eqref{ineqMk}.

\med
Let us finally prove {\it iii)}. The inequality $M_k(\psi) \le M_k(\psi_+)$ is trivial. To prove the converse inequality, we observe that for every $x_1,x_2,\dots x_k$ in $\R^d$, it holds
$$\frac1k\sum_{i=1}^k\psi_+(x_i)-c_k(x_1,\dots,x_k)\le\frac1k\sum_{i=1}^k\psi(y_i)-{\tilde c_k}(y_1,\dots,y_k)\le M_k(\psi)\;,$$
where $y_i= x_i$ whenever $\psi(x_i)\ge 0$ whereas $y_i=\omega$ otherwise, being ${\tilde c_k}$ the natural extension of $c_k$ to $(\R^d\cup\{\omega\})^k$. One readily checks that $c_k(x_1,\dots,x_k)\ge{\tilde c_k}(y_1,\dots,y_k)$ while $\sum_{i=1} ^k \psi_+(x_i)\le \sum_{i=1} ^k \psi(y_i)$ since $\psi(\omega)=0$.
\end{proof}

From now on, we will use for $\Cbar(\rho)$ (resp. for $\C_k(\rho_k)$)
given by \eqref{cdual} (resp.\eqref{Cextend}) the duality formulae rewritten in a condensed form as follows.
\begin{prop} \ For every $\rho\in\PP^-$, the following equalities hold:
\begin{align}\label{compCbar}
\Cbar(\rho)= \sup_{\f\in C_0}\left\{\int\f\,d\rho-M_N(\f)\right\}.\end{align}
In other words, $\Cbar$ is the Fenchel conjugate of $M_N$ in the duality between $C_0$ and the space of bounded measures. In addition, for every $k\le N$, we have
\be
\label{compCcalk}\C_k(\rho_k)=\sup_{\f\in C_0}\left\{\int\f\,d\rho_k-M_k(\f)\|\rho_k\|\right\}.
\ee
\end{prop} 
\begin{proof}\ For \eqref{compCbar}, we use the first equality in \eqref{cdual} with the change of variables
 $\f= \psi-\psi_\infty$, taking into account that $\psi\in \A$ is equivalent to $\psi_\infty\le - M_N(\f)$.
For \eqref{compCcalk}, it is enough to apply  \eqref{compCbar} replacing $N$ by $k$ and $\rho$ by the probability
$ \frac{\rho}{\|\rho\|}$.
\end{proof}
By \eqref{plus}, it turns out that the supremum in \eqref{compCbar} and \eqref{compCcalk} are unchanged when they are restricted to nonnegative functions $\f\in C_0$. In particular an optimal potential (if it exists) or any maximizing sequence can be assumed to be nonnegative. 

\begin{theo}\label{prop:cns}
Let $\rho\in\PP-$ such that $\|\rho\|>1/N$. Let $\{\rho_k\}$ a decomposition such that 
$$\sum_{k=1}^N \frac{k}{N}\rho_k=\rho,\ \sum_{k=1}^N\|\rho_k\|\le1.$$
Then $\{\rho_k\}$ is optimal in \eqref{rhok} and $\f$ is optimal in \eqref{compCbar} (respectively $(\f_n)$ is a maximizing sequence) if and only if the three following conditions hold:
\begin{itemize}
\item [i)] \ $\ds\sum_{k=1}^N \|\rho_k\| =1$,
\item [ii)] \ For all $k$, $\ds \frac{k\f}{N} $ is optimal (resp. $\ds \frac{k\f_n}{N} $ is a maximizing sequence) in \eqref{compCcalk}
\item [iii)] \ $ M_k(\frac{k\, \f}{N})= M_N(\f)$ (resp. $ M_N(\f_n)\!-\!M_k(\frac{k\f_n}{N})\to 0$) holds whenever it exists $l\le k$ such that $\|\rho_{l}\|>0$.
\end{itemize}
\end{theo}

\begin{proof} For any admissible pair $(\{\rho_k\},\f)$, we have 
$$\sum_k\C_k(\rho_k)\ge\int\f\,d\rho-M_N(\f).$$
Thus the optimality arises as soon the previous inequality becomes an equality.
We compute:
\begin{align}
\sum_k\C_k(\rho_k)-\left(\int\f\,d\rho-M_N(\f)\right)=
&\sum_k\nonumber\bigg(\C_k(\rho_k)-\int\frac{k\f}{N}\,d\rho_k+M_k(\frac{k\f}{N})\|\rho_k\|\bigg)\label{identity}\\
&+\sum_k\left(M_N(\f)-M_k(\frac{k\f}{N})\right)\|\rho_k\|\\
&+M_N(\f)\left(1-\sum_k\|\rho_k\|\right).\nonumber
\end{align}
By \eqref{compCcalk} and \eqref{ineqMk}, we discover that the right hand side of \eqref{identity} consists of the sum of three nonnegative terms.
Thus the left hand side vanishes if and only if all these three terms vanish that is to say i), ii) and iii) hold simultaneously. Note that for iii) we use the monotonicity property \eqref{ineqMk} allowing to pass from index $l$ to any $k\ge l$ and Remark \ref{nonnul} where we noticed that $M_N(\f)>0$ (resp. $\liminf_{n\to\infty} M_N(\f_n)>0$ in case of a maximizing sequence $(\f_n)$).
\end{proof}

\begin{rema} \label{nonnul} Any optimal $\f$ satisfies $M_N(\f)>0$ since otherwise, by \eqref{def:Mk}, the inequalities $\sup \f=N M_1(\frac{\f}{N})\le N M_N(\f)\le0$ would imply that 
$$\Cbar(\rho)=\int\f\,d\rho-M_N(\f)\le(N\|\rho\|-1)M_N(\f)=0$$
which is excluded since $\Cbar(\rho)>0$ if $\|\rho\|>1/N$ (see Proposition \ref{cbzero}). On the same way, if $(\f_n)$ is an optimal sequence, we infer that $\liminf_{n\to\infty} M_N(\f_n)>0$.
\end{rema}
\begin{rema} \label{rhoknonnul}
Note that, in Theorem \ref{prop:cns}, the condition $\sum_{k=1}^N \|\rho_k\| =1$ holds for any $\{\rho_k\}$ optimal in \eqref{compCbar} since  there always exists a maximizing sequence $(\f_n)$ for the dual problem.
Next we observe that, if $\overline{ k}$ denotes the integer part of $N\|\rho\|$, then
the equality $N \|\rho\|= \sum_{k=1}^N k \|\rho_k\|$ and $\sum_{k=1}^N \|\rho_k\| = 1$ imply that there exist at least two integers $l_- \leq \overline{k} \leq l_+$ such that $\|\rho_{l_\pm}\|>0$. Accordingly the assertion iii) of Theorem \ref{prop:cns} includes all values $k > N\|\rho\|-1$. 
\end{rema}

\section{Existence of a Lipschitz potential for the relaxed cost}\label{dualexist}

The main result of this Section is the existence of an optimal potential for the relaxed cost $\Cbar(\rho)$. 
Such an existence result is already known for $\|\rho\|=1$ under a suitable low concentration assumption on 
the probability $\rho$ (see \cite{bcdp17}, Theorem 3.6 or \cite{cdms18} for the sharp constant). More precisely, for every $\rho\in \PP^-$, we define
$$ K(\rho)=\sup\big\{\rho(\{x\})\ :\ x\in\R^d\big\}.$$
Then if $K(\rho)< \frac1{N}$, it is shown in \cite{bcdp17,cdms18} that there exists an optimal continuous bounded Lipschitz potential $u\in \B$. As a preamble, we
 prove that in fact this optimal potential $u$ can be chosen in the subclass $\A$, i.e. $u$ is constant at infinity.
 
\begin{prop}
Let $\rho\in\PP$ and let $u \in \B$ be an upper bounded optimal potential for $\rho$. Then  ${\tilde u}:= \max\{u, u^\infty\}$ is still an optimal potential for $\rho$.
In particular if $u$ is continuous (resp. Lipschitz continuous),
the optimal potential ${\tilde u}$ belongs to $\A$ (resp. to $\A\cap Lip$).
\end{prop}

\begin{proof} It is enough to check that $\tilde u$ is admissible which follows from Lemma \ref{troncation}.
\end{proof}

Now we are going to extend this existence and regularity result to sub-probabilities under the following assumption on $\rho$:
\be\label{finitecost}
\|\rho\|<1,\qquad\exists\delta>0\ :\ \Cbar((1+\delta)\rho)<+\infty
\ee

\begin{theo}\label{existencedual}
Let $\rho\in\PP^-$ satisfying \eqref{finitecost} for a given $\delta>0$. Then there exist a Lipschitz optimal potential $u\in C_0\oplus\R$ solving \eqref{cdual}. The Lipschitz constant of $u$ depends only on $\delta$. Furthermore any solution to \eqref{cdual} coincides with a Lipschitz one on the support of $\rho$.
\end{theo}

\begin{rema} The finiteness condition in \eqref{finitecost} is fulfilled in particular if the concentration satisfies $K(\rho) <\frac1N$. Indeed, we may chose $\delta$ and a smooth density measure $\nu$ so that $(1+\delta) \rho+\nu$ is a probability measure with a concentration still lower than $\frac1{N}$, thus with finite cost. By applying Corollary \ref{smallermass}
we infer that
$$\Cbar\big((1+\delta)\rho\big)\le\Cbar\big((1+\delta)\rho+\nu\big)=C\big((1+\delta)\rho+\nu\big)<+\infty.$$
\end{rema}

The proof of Theorem \ref{existencedual} is quite involved and is given in the remaining part of this Section. First we need to fix some notations and give some preparatory results which are collected in the next subsection.

\subsection{Preliminary results}

We recall the expression \eqref{compCbar} for $\Cbar(\rho)$ that we are using. The existence of an optimal potential $u$ amounts to find a function $\f\in C_0$ such that
\begin{equation}\label{achievement}
\Cbar(\rho) = I_N(\f) \quad \text{where}\quad I_N(\f): =\int\f\,d\rho-M_N(\f)
\end{equation}
where we recall
$$M_N(\f)=\sup\left\{\frac1N\sum_{i=1}^N\f(x_i)-c_N(x_1,\dots,x_N)\ :\ x_i\in\R^d\right\}$$
We notice that the definition of $M_N(\f)$ above can be obviously extended to any upper bounded Borel function. Accordingly we have very useful properties which are given in the two next Lemmas.

\begin{lemm} \label{DeltaN} 
Let $\f:\R^d\to\R$ an upper bounded Borel function and set $\f^\infty:=\limsup_{|x|\to+\infty}\f(x)$. Then the following inequalities hold:
\begin{equation}\label{supestiMN}
\frac1{N} \sup \f + \frac{N-1}{N}\f^\infty \le M_N(\f)\le \sup \f.
\end{equation}
\be\label{finfty=0}
\frac1{N} \f^\infty +  M_{N-1}\Big(\frac{N-1}{N}\f\Big) \le  M_N(\f).
\ee
\end{lemm}

\begin{proof}
The inequality $M_N(\f)\le\sup\f$ is trivial. On the other hand, it holds for every $x=(x_i)$ in $(\R^d)^N$:
$$M_N(\f)\ge\frac{\f(x_1)}{N}+\sum_{j=2}^N\frac1{|x_1-x_j|}+\left(\frac1{N}\sum_{i=2}^N\f(x_i)-c_{N-1}(x_2,\dots,x_N)\right)$$
By sending all points $x_i$ (with $i\ge2$) to infinity and then taking the supremum in $x_1$, we deduce the first inequality in \eqref{supestiMN}. On the opposite, if we send first $x_1$ to infinity and then optimize with repect to all $x_i$ with $i\ge2$, we get \eqref{finfty=0}
\end{proof}

A consequence of \eqref{supestiMN} is that for elements $\f\in C_0^+$, $M_N(\f)$ is equivalent to the uniform norm. In the sequel we will denote
\be\label{def:DeltaN}
\Delta_N(\f):=M_N(\f)-M_{N-1}\Big(\frac{N-1}{N}\f\Big).
\ee
By \eqref{ineqMk}, we have $\Delta_N(\f)\ge 0$ for every $\f\in C_0$. Now if $\f$ is   a nonnegative
element of $C_b$, a every useful recipe in order to show that $\f$ belongs to $C_0$ is to verify that $\Delta_N(\f)=0$ (just by applying  by \eqref{finfty=0}). 

\begin{lemm} \label{familyMN} 
Let $\f_n:\R^d\to\R$ be a family of Borel functions such that
$$\f_{n+1}\ge\f_n,\qquad\f:=\sup_n\f_n\le C,$$
where $C$ is a suitable constant. Then, for every $k\in\N$, it holds
$$\lim_{n\to\infty} M_k(\f_n)=\sup_n M_k(\f_n)=M_k(\f).$$
\end{lemm}

\begin{proof}
Clearly $M_k(\f_{n})\le M_k(\f_{n+1})\le M_k(\f)$, so that $\lim_n M_k(\f_n)\le M_k(\f)$. On the other hand, as $\f_n\to\f$ pointwise, we have for every $x=(x_i)\in(\R^d)^k$:
$$\liminf_n M_k(\f_n)\ge\lim_n\frac1k\sum_{i=1}^k\f_n(x_i)-C_k(x)=\sum_{i=1}^k\f(x_i)-C_k(x),$$
hence $\liminf_n M_k(\f_n)\ge M_k(\f)$ by optimizing with respect to $x$.
\end{proof}

Next, for every upper bounded Borel function $\f$, we introduce the new function:
$$ [M_N\f](x):=\sup \left\{ \frac1{N} \sum_{i=1} ^N \f(x_i) - c_N(x_1,\dots,x_N)\ :\ x_1=x, (x_2,\dots,x_N)\in(\R^d)^{N-1}\right\}.$$
By construction, it holds $M_N(\f) = \sup \{[M_N\f](x)\ :\ x\in\R^d\}$. It turns out that, for $\f\in C_0$, the limit of $[M_N\f]$ at infinity is nothing else but $M_{N-1}\Big(\frac{N\!-\!1}{N}\f\Big)$. A key argument in the proof of Theorem \ref{existencedual} is the introduction of the regularization of $\f$ defined as follows:
\begin{equation}\label{def:hatf}
\hat\f (x) \ =\ \f(x) +\ N\,\left( M_{N-1}\Big(\frac{N\!-\!1}{N}\f\Big) - [M_N\f](x)\right)
\end{equation}
It is easy to check that $\hat\f $ can be rewriten in the following form 
\be\label{def2:hatf}
\hat\f(x)=\inf_{x_2,x_3,\dots,x_N}\left\{N\,c_N(x,x_2,\dots,x_N)-\sum_{i=2}^N\f(x_i)\right\}+N\,M_{N-1}\Big(\frac{N-1}{N}\f\Big)\;.
\ee
Here we used an additional constant in order to preserve the vanishing condition at infinity
(see the Lemma \ref{propfhat} hereafter). The next fundamental Lipschitz estimate enlights the regularization effect of the map $\f\mapsto\hat\f$.
 
\begin{prop}\label{prop:regularpot} For every $R>0$, there exists a constant $\gamma_N(R)$ such that:
\be\label{equiLip}
\{\hat{\f}\ :\ \f\in C_0,\ \f\le R\}\subset Lip_{\gamma_N(R)}(\R^d).
\ee
\end{prop}

\begin{proof} Let $\f\in C_0$. 
Recalling the expression \eqref{def2:hatf} for $\hat\f$, a preliminary estimate is the following:
$$\inf_{x_2,\dots,x_N}\left\{Nc_N(x,x_2,\dots,x_N)-\sum_{i=2}^N\f(x_i)\right\}\le0,$$
which is obtained taking $x_2,\dots, x_n$ arbitrarily away from $x$ and from each others.

A more delicate estimate is this: let $x\in\R^d$ and let $\eps >0$. Then there exists $\eta= \eta(\eps)>0$ such that for all $\overline x_2,\dots,\overline x_N$ which almost realize $\hat{\f}(x)$ in the sense that
$$Nc_N(x,\overline x_2,\dots,\overline x_N)-\sum_{i=2}^N\f(\overline x_i)+NM_{N-1}\Big(\frac{N-1}{N}\f\Big)\le\hat{\f}(x)+\eps$$
and
$$Nc_N(x,\overline x_2,\dots,\overline x_N)-\sum_{i=2}^N\f(\overline x_i)\le\eps$$
it holds
$$|x-\overline x_i|\ge\eta,\qquad i=2,\dots,N.$$
Hence the $\overline x_i$'s need to be at least at distance $\eta$ from $x$ where $\eta$ does not depends on $x$. In fact for all $i\in\{2,\dots,N\}$,
$$\eps\ge Nc_N(x,\overline x_2,\dots,\overline x_N)-\sum_{i=2}^N\f(\overline x_i)\ge N \frac{1}{|x-\overline x_i |}-(N-1)R$$
so that
$$|x-\overline x_i|\ge\frac{N}{\eps+(N-1)R}.$$
In particular we may choose $\eta(\eps)=\frac{N}{\eps+(N-1)R}$ and for all $\eps\in(0,1]$ we have $\eta(\eps)\ge\frac{N}{1+(N-1)R}$.

We may now make the Lipschitz estimates for $\hat\f$. Let $x\in\R^d$, let $\eps \leq 1$ and let $y$ be such that $|y-x|\le\frac{\delta}{4}$ we choose $\overline x_i$ for $i\in \{2, \dots N\}$ which almost realize $\hat{\f}(x)$ in the sense above. We use the $\overline x_i$ in both formulas for $\hat{\f}(x)$ and $\hat{\f}(y) $ to obtain
\begin{eqnarray*}
\hat{\f}(y)-\hat{\f}(x)-\eps&\le&N c_N(y,\overline x_2,\dots,\overline x_N)-\sum_{i=2}^N\f(\overline x_i)-Nc_N(x,\overline x_2,\dots,\overline x_N)+\sum_{i=2}^N\f(\overline x_i)\\
&=&N\sum_{i=2}^N \left(\frac{1}{|y-\overline x_i|}-\frac{1}{|x-\overline x_i|}\right) \leq N \sum_{i=2}^N \frac{|x-y|}{|\xi_i-\overline x_i|^2} \\
&\leq&\frac{N(N-1)16}{9\delta^2}|x-y|\le\frac{(N-1)(1+(N-1)R)^216}{9N}|x-y|,
\end{eqnarray*}
where the inequality of the second line, holding for a suitable $\xi_i$, follows from Lagrange intermediate value theorem applied to the functions $1/|\cdot-\overline x_i|$. This is allowed since $|x-\overline x_i|>\eta$ and $|x-y|\le\eta/4$; in particular $|\xi_i-\overline x_i|\ge3\eta/4$. From the inequalities above one gets a Lipschitz estimate for $\hat{\f}$ in $B(x,\eta/4)$ with a constant independent of $x$ that we denote by $\gamma_N(R)$. Then, clearly we have a global Lipschitz estimate with the same constant.
\end{proof}

\begin{rema} The second addendum in \eqref{def2:hatf} is estimated as
$$M_{N-1}\left(\frac{N-1}{N}\f\right)=\sup_{x_2,\dots,x_N}\left\{\sum_{i=2}^N \frac{1}{N}\f(x_i)-c_{N-1}(x_2, \dots,x_N)\right\}\le\frac{N-1}{N} R,$$
which is obtained by the fact that $c_{N-1}$ is positive and $\f\le R$. All in all we have $\hat{\f}\le(N-1)R$
\end{rema}

We conclude this subsection with a crucial technical result.

\begin{lemm}\label{propfhat} 
Let $\f\in C_0$, $\hat\f$ defined by \eqref{def:hatf} and $\Delta_N(\f)$ defined in \eqref{def:DeltaN}. Then
\begin{itemize}
\item[i)]$\hat\f$ belongs to $C_0$.
\item[ii)]The function $\psi=(1-\frac1N)\f+\frac1N\hat\f$ satisfies:
\be\label{identity2}
M_N(\psi)=M_{N-1}\Big(\frac{N-1}{N}\f\Big).
\ee
\be\label{energy}
I_N(\psi)\ge I_N(\f)+(1-\|\rho\|)\Delta_N(\f).
\ee
\be\label{bounds}
\psi\ge\f-\Delta_N(\f).
\ee
\end{itemize}
\end{lemm}

\begin{proof}
Let us prove {\it i)}. By Proposition \ref{prop:regularpot}, we know already that $\hat\f$ is Lipschitz continuous. Owing to \eqref{def:hatf}, we have only to show that
$$\lim_{|x|\to+\infty}[M_N\f](x)=M_{N-1}\Big(\frac{N\!-\!1}{N}\f\Big).$$
First, as $c_N(x_1,x_2,\dots,x_N)\ge c_{N-1}(x_2,\dots,x_N)$, we deduce that:
\begin{align*}
[M_N\f](x) &\le \frac{\f(x)}{N}+\sup_{ (x_2,\dots , x_N)} \left\{\frac1{N} \sum_{i=2} ^N \f(x_i)-c_{N-1}(x_2,\dots,x_N)\right\}\\
&=\frac{\f(x)}{N}+M_{N-1}\Big(\frac{N\!-\!1}{N}\f\Big)
\end{align*}
Thus, as $\f\in C_0$, we have
$$\limsup_{|x|\to \infty} [M_N\f](x)\le M_{N-1}\Big(\frac{N\!-\!1}{N}\f\Big).$$
For the converse inequality, we observe that, for every $(x_2, \dots , x_n)$, it holds
$$[M_N\f](x)\ge\frac{\f(x)}{N}-\sum_{j=2}^N \frac1{|x-x_j|} + \frac1{N} \sum_{i=2} ^N \f(x_i)-c_{N-1}(x_2,\dots,x_N).$$
hence the conclusion by sending first $|x|$ to infinity and then optimizing with respect to $x_2,\dots ,x_N$.

\med We prove now {\it ii)}. First the lower bound of $\psi$ given in \eqref{bounds} is obtained by recalling that $[M_N\f](x)\le M_N(\f)$. Then we infer that:
$$ \hat \f(x) \ge \f(x) + N \Big( M_{N-1}\Big(\frac{N\!-\!1}{N}\f\Big) - M_N(\f)\Big) = \f(x) - \Delta_N(\f).$$
In order to show \eqref{identity2}, we observe that, by the definition of function $[M_N\f]$, we have
$$\sum_{i=1}^{N} [M_N\f](x_i)\ge\sum_{i=1}^N \f(x_i)-N\,c_N(x_1,x_2,\dots,x_N).$$
By applying the definitions of $\psi$ with $\hat\f$ given by \eqref{def:hatf}, it follows that, for every $x=(x_1,x_2,\dots,x_N)\in(\R^d)^N$:
\begin{align*}
\sum_i\psi(x_i)&=\sum_i\f(x_i)+N\,M_{N-1}\Big(\frac{N-1}{N}\f\Big)-\sum_i[M_N\f](x_i)\\
&\le N\,M_{N-1}\Big(\frac{N-1}{N}\f\Big)+N\,c_N(x)
\end{align*}
Therefore
$$\frac1{N}\sum_i \psi(x_i)- c_N(x)\le M_{N-1}\Big(\frac{N\!-\!1}{N}\f\Big)$$
and the inequality $M_N(\psi)\le M_{N-1}\big(\frac{N-1}{N}\f\big)$ follows by maximizing with respect to $x$. The converse inequality holds true since, by \eqref{bounds}
$$M_N(\psi)\ge M_N(\f)-\Delta_N(\f)=M_{N-1}\Big(\frac{N-1}{N}\f\Big).$$
Eventually we infer also \eqref{energy} as a consequence of \eqref{identity2} and \eqref{bounds}.
\end{proof}

\subsection{Proof of Theorem \ref{existencedual}}

We proceed in several steps.
 
\med{\bf Step 1. }{\it Let $\delta$ as given by the assumption \eqref{finitecost}. Then there exists $R=R(\delta) >0$ such that:
$$\Cbar(\rho)=\sup\left\{I_N(\f)\ :\ \f\in C_0(\R^d,[0,R])\right\}.$$}
Indeed, by \eqref{plus}, we have $M_N(\f_+)=M_N(\f)$, thus $I_N(\f_+)\ge I_N(\f)$ for every $\f\in C_0$. Therefore the supremum of $I_N(\f)$ is unchanged if we restrict to $\f\in C_0^+$. On the other hand, for every given $\eps>0$, we may restrict the supremum to the subclass
$$\A_\eps:=\left\{\f\in C_0^+\ :\ I_N(\f)\ge\Cbar(\rho)-\eps\right\}.$$
Since
$$\Cbar((1+\delta)\rho)\ge(1+\delta)\int\f\,d\rho-M_N(\f),$$
we deduce that, for every $\f\in\A_\eps$, it holds:
$$\Cbar((1+\delta)\rho)-(1+\delta)\Cbar(\rho)\ge\delta\,M_N(\f)-\eps\ge\frac{\delta}{N}\,\sup\f-\eps.$$
Therefore $\A_\eps\subset C_0(\R^d,[0,R])$ for small $\eps$, provided
$$R>\frac{\Cbar((1+\delta)\rho)-(1+\delta)\Cbar(\rho)}{N\delta}.$$

\med{\bf Step 2. }{\it For every $\eps>0$, there exists $\psi\in C_0(\R^d,[0,NR])$ such that
\begin{equation}\label{epssolution}
I_N(\psi)\ge\Cbar(\rho)-\eps,\qquad M_N(\psi)\le R,\qquad{\rm Lip}(\psi)\le \gamma_N(NR).
\end{equation}}

The existence of $\psi$ satifying \eqref{epssolution} will be derived after designing a suitable sequence $(u_n)$ in $C_0$. We start with an element $\f_\eps\in C_0(\R^d;[0,R])$ such that $I_N(\f_\eps)>\Cbar(\rho)-\eps$ as given in Step 1. Then we define a sequence $(u_n)$ as follows:
$$u_0=\f_\eps,\qquad u_{n+1}=\frac1N\hat u_n+\frac{N-1}{N}u_n.$$ 
Applying Proposition \ref{propfhat} we get
\begin{align}
I_N(u_{n+1})&\ge I_N(u_n)+(1-\|\rho\|)\Delta_N(u_n)\label{energyn}\\
u_{n+1}&\ge u_n-\Delta_N(u_n)\label{psimonotone}\\
M_N(u_n)&\ge M_{N-1}\Big(\frac{N-1}{N}u_n\Big) = M_N(u_{n+1}).\label{MNmonotone}
\end{align}

\med From \eqref{energyn} follows that $I_N(u_n)$ is non-decreasing. Since $I_N(\f_\eps)\le I_N(u_n) \le \Cbar(\rho)$, 
its limit satisfies:
\begin{equation}\label{limiteI_N}
\Cbar(\rho) -\eps \ <\ \lim_n I_N(u_n) \le\ \Cbar(\rho) \ .
\end{equation}
Now we use the condition $\|\rho\|<1$ to infer from \eqref{energyn} that 
$$\sum_{n=1}^\infty \Delta_N(u_n)\le\frac{\eps}{1-\|\rho\|}<+\infty.$$
Let us denote by $\eps_n:=\sum_{k\ge n}\Delta_N(u_k)$ the remainder of the series; we see from \eqref{psimonotone} that $ v_n= u_n - \eps_n$ is monotone non-decreasing. Therefore $u_n$ and $v_n$ share the same point-wise limit $u(x)$ which at least is a lower semicontinuous function. Next we can derive in a straightforward way a uniform upper bound for the $u_n$ by applying the monotonicity property \eqref{MNmonotone}. Indeed, according to the choice $u_0=\f_\eps$ for the initial term which satisfies $\sup u_0\le R$, we have
\begin{equation}\label{uniformbound}
\frac1{N}\sup u_n\le M_N( u_n)\le M_N(u_0)\le R.
\end{equation}
Then we may apply to $(v_n)$ the continuity property on monotone sequences given in Lemma \ref{familyMN} for $k=N-1$ and $k=N$, noticing that $M_k(u_n) = M_k(v_n) +\eps_n$:
$$M_N(u_n)\to M_N(u),\qquad M_{N-1}(\frac{N-1}{N}u_n)\to M_{N-1}(\frac{N-1}{N}u)\;.$$
It follows that
$$\Delta_N(u)=\lim_n \Delta_N(u_n)=0.$$
As a consequence of \eqref{finfty=0}, we deduce that
$$u^\infty=\limsup_{|x|\to\infty}u(x)\le0\, .$$

Next, in order to gain the Lipschitz regularity of $u$, we are going to apply Proposition \ref{prop:regularpot} to the sequence $(u_n)$. By construction, we have $\hat u_n - u_n= N(u_{n+1} - u_n)$. Therefore $\hat u_n- u_n \to 0$ and $\hat u_n\to u$ poitwise on $\R^d$. As a consequence of \eqref{equiLip} $(\hat u_n)$ is equi-Lipschitz with constant $\gamma_N(NR)$. By Arzel\'a-Ascoli's Theorem, it converges to $u$ uniformly on compact subsets of $\R^d$. Its limit $u$ satisfies as well $\sup u\le NR$ and it is Lipschitz continuous with the constant $\gamma_N(NR)$.

\med
Eventually we claim that the function $\psi= u_+$ satisfies the three requirements in \eqref{epssolution}. Indeed, $u^+(\omega)\le0$ implies that $\psi$ is an element of $C_0(\R^d;[0,NR]$. It has the same Lipschitz constant $\gamma_N(NR)$. In addition, by Lemma \ref{plus} and \eqref{uniformbound}, we have $M_N(\Psi)=M_N(u)\le R$. Eventually, by monotone convergence,
we have:
$$\lim_n I_N(u_n)=\lim_n\int u_n\,d\rho-\lim_n M_N(u_n)=\int u\,d\rho-M_N(u)=I_N(u),$$
and the first condition in Claim \eqref{epssolution} follows from \eqref{limiteI_N}. 

\med{\bf Step 3. }{\it There exists a sequence $(\f_n)\in C_0(\R^d,[0,R])$ and a function $\f \in C_0(\R^d,[0,NR])$ with ${\rm Lip}(\f)\le \gamma_N(NR)$ such that
\begin{equation}\label{final}
\f_{n+1} \ge \f_n,\qquad I_N(\f_n)\to\Cbar(\rho),\qquad\sup_n\f_n=\f.
\end{equation}}
By applying step 2 for $\eps=\frac1{n}$, we obtain a sequence $(\psi_n)$ in $C_0(\R^d;[0,NR])$ with a uniform Lipschitz constant $\gamma_N(NR)$ and such that $I_N(\psi_n)\to\Cbar(\rho)$. By Ascoli-Arzela's Theorem and possibly after passing to a suitable subsequence, we may assume that $\psi_n$ converges uniformly on compact subsets to a Lipschitz continuous $\f\in C(\R^d;[0,NR])$. 
At this point, we would need also a uniform convergence on the whole $\R^d$ in order to conclude that $\f$ vanishes at infinity. To avoid this difficulty, we turn to another sequence in $C_0(\R^d,[0,NR])$, namely $(\f_n)$ defined by
$$\f_n:=\inf\left\{\psi_m\ :\ m\ge n\right\}.$$
Clearly the pointwise convergence $\psi_n\to\f$ implies that $\f_n$ converges increasingly to $\f$. As $M_N(\f_n) \le M_N(\psi_n)$, we have
$$I_N(\f_n)\ge I_N(\psi_n)-r_n,\qquad r_n=\int(\psi_n-\f_n)\,d\rho,$$
where $r_n\to0$ by dominated convergence. Therefore $(\f_n)$ is a maximizing sequence for \eqref{compCbar} and by applying the assertion iii) of Theorem \ref{prop:cns} for $k=N-1$ (see Remark \ref{rhoknonnul}), we deduce that
$$M_N(\f_n)-M_{N-1} \left(\frac{N-1}{N} \f_n\right)\to0.$$
Thus, again by the monotonicity property of Lemma \ref{familyMN}, we are led to the equality 
$$M_N(\f)-M_{N-1}\left(\frac{N-1}{N}\f\right)=0$$
from which follows that $\f^\infty\le 0$ (see Lemma \ref{DeltaN}). As $\f$ is continuous nonnegative, we conclude that $\f\in C_0$ thus \eqref{final}.

\med
\begin{proof}[Concluding the proof.]
The $\f$ constructed in Step 3 obviously satisfies \eqref{achievement}. Indeed the convergence $\f_n\to\f$ is strong in $C_0$ (as a consequence of Dini's Theorem on the compact set $\R^d\cup\{\omega\}$) and therefore, recalling that the map $M_N:C_0\to\R$ is Lipschitz continuous (see Lemma \ref{prop:Mk}), it holds $M_N(\f_n)\to M_N(\f)$. Thus $u=\f-M_N(\f)$ is an optimal potential for the dual problem \eqref{cdual}. Its Lipschitz constant is not larger than $\gamma_N(NR)$ given by Proposition \ref{prop:regularpot}, being $R=R(\delta)$ given in Step 1.

Eventually let be $v$ be another solution to \eqref{cdual}. Then $v=\tilde{\f}-M_N(\f)$ for an element $\tilde{\f}\in C_0$ solving \eqref{achievement}. Thanks to \eqref{energy}, the function 
$\psi=(1-\frac1N)\tilde{\f}+\frac1N \widehat{\tilde{\f}}$ introduced in Lemma \ref{propfhat} satisfies:
$$I_N(\psi)\ge I_N(\tilde{\f})+(1-\|\rho\|)\,\Delta_N(\tilde{\f}).$$
The optimality of $\tilde{\f}$ implies that $I_N(\psi)=\ I_N(\tilde{\f})$ and that $\Delta_N(\tilde{\f})=0$, thus $\psi\ge \tilde{\f}$ thanks to \eqref{bounds}. It follows that $\psi=\tilde{\f}= \hat \psi$ holds $\rho$ a.e., hence on $\spt(\rho)$ by continuity. As $\sup \tilde{\f} \le R=R(\delta)$ by step 1, then $M_N(\psi)\le M_N(\tilde{\f})$ implies that $\sup\psi\le NR$ and, by applying\eqref{equiLip}, $\psi$ is Lipschitz with constant $\gamma_N(NR)$ while it coincides with $\tilde{\f}$ on $\spt(\rho)$. 
\end{proof}

\section{Quantization of relaxed minimizers}\label{relaxedmass}

In this Section we focus on the relaxed problem mentioned in the introduction namely
\be\label{minrelax}
\min\left\{\Cbar (\rho)-\int V\,d\rho\ :\ \rho\in\PP^-\right\}\ ,
\ee
where $V$ is a given potential in $C_0$. Note that the infimum above would blow-up to $-\infty$ if $V$  
is not upper bounded, as for instance in the case of Coulomb potential. 
The existence of solutions to \eqref{minrelax} in $\PP^-$ is straightforward as we minimize a convex lower semicontinuous functional on the weakly* compact set $\PP^-$. 
On the other hand, as the minimum in \eqref{minrelax} agrees with that of
\be \label{nonrelaxed} \inf\left\{C (\rho)-\int V\,d\rho\ :\ \rho\in\PP\right\}, \ee
any solution $\rho\in\PP$ to \eqref{minrelax} is also a solution to \eqref{nonrelaxed} and vice-versa.

\medskip
We pay now attention to the set of minimizers 
$$\mathcal{S}_V=\left\{\rho\in\PP^-\ :\ \rho \ \text{solves \eqref{minrelax}} \right\}.$$
As by \eqref{compCbar} $ \Cbar(\rho)$ agrees with the Fenchel conjugate of $M_N$, we may interpret  $\mathcal{S}_V$ in the language of convex analysis as
the sub-differential of $M_N$ at $V$, i.e.
$$\mathcal{S}_V=\left\{\rho\in\PP^-\ :\ \Cbar(\rho)-\int V\,d\rho+M_N(V)\le 0 \right\}.$$
In particular  $\mathcal{S}_V$ is a convex weakly* compact subset of $\PP^-$. Note that in general $\mathcal{S}_V$
is not a singleton as the functional $\Cbar$ is not stricly convex. 
Besides we observe that the minimum value of \eqref{minrelax}  is strictly negative unless
 the positive part of $V$ vanishes. Indeed , by considering competitors $\rho$ such that $\|\rho\|\le \frac1{N}$ (thus $\Cbar(\rho)=0$ by Proposition \ref{cbzero}), we have the following estimate
\be\label{miniesti}
\inf\big\{\eqref{minrelax}\big\}\ \le\ -\frac1N\,\sup V^+.
\ee


One of the major questions in the ionization problem, as developed for instance in \cite{fnvdb18,fnvdb18b,solo03,solo91} in a much more complex case,
is to determine conditions on the potential $V$ under which \eqref{minrelax} admits solutions in $\PP$. We give here a sufficient condition.
\begin{theo}\label{optimalproba}
Assume that the potential $V$ satisfies the condition
\be \label{strictMN}
M_N(V)>M_{N-1}\Big(\frac{N-1}{N} V\Big).
\ee 
Then all solutions $\rho$ to \eqref{minrelax} satisfy $\|\rho\|=1$, that is $\mathcal{S}_V \subset\PP$.
Moreover the supremum defining $M_N(V)$ in \eqref{def:Mk} is a maximum.
\end{theo}
\begin{proof}
Assume that there exists a solution $\rho$ such that $\|\rho\|<1$. Then, as $V$ is a solution to the dual problem associated with $\rho$, we may apply the optimality conditions derived in Theorem \ref{prop:cns}. Let $\{\rho_l\}$ be optimal in \eqref{rhok}. Then by Remark \ref{rhoknonnul}, we know that it exists at least one index $l\le N-1$ such that $\|\rho_l\|>0$ and therefore, condition iii) is satisfied for $k=N-1$. Thus $M_N(V)=M_{N-1}(\frac{N-1}{N} V)$, in contradiction with our assumption.
For the last statement, we consider a point $\bar x\in X^N$ (recall that $X=\R^d\cup \{\omega\}$) such that:
$$ M_N(V)\  =\sup_{x\in (\R^d)^N} \frac1{N} \sum_{i=1}^N V(x_i) - c_N(x)\ =\ \frac1{N} \sum_{i=1}^N V(\overline{x_i}) - {\widetilde c_N}(\bar x)\ ,$$
being ${\widetilde c_N}$ the natural upper semicontinuous extension of $c_N$ to $X^N$ (see \eqref{ctil}). 
Such an optimal $\bar x$ exists since we maximize an u.s.c. function on a compact set.
If the infimum is not reached in  $(\R^d)^N$, that means that $\overline{x_i}=\omega$ for at least one index $i$
for instance $i=N$ and we are led to
$$  M_N(V) = \sum_{i=1}^{N-1} V(\overline{x}_i) - {\widetilde c_{N-1}}(\overline{x}_1,\overline{x}_2,\dots, \overline{x}_{N-1} ) \le M_{N-1}\Big(\frac{N-1}{N} V\Big) \ ,$$
in contradiction with \eqref{strictMN}.  
\end{proof}

In view of Theorem \ref{optimalproba}, a meaningful issue is now to understand what happens when equality $M_N(V)= M_{N-1}(\frac{N-1}{N} V)$ holds. To that aim it is useful to introduce:
\begin{equation}\label{def:kbar}
k_N(V):=\max\left\{k\in\{1,2,\dots, N\}\ :\ M_k\Big(\frac{k}{N}\,V\Big)>M_{k-1}\Big(\frac{k-1}{N}\,V\Big)\right\}.
\end{equation}
Here we set by convention $M_0(0)=0$ so that $k_N(V)$ is well defined if $V^+$ does not vanish. Otherwise we set $k_N(V)=0$.

\begin{coro}\label{slowexists}
Let $V$ be a potential $V\in C_0^+$ such that:
$$\beta:=\limsup_{|x|\to+\infty}|x|V(x)>0.$$
Then the condition \eqref{strictMN} is satisfied whenever 
\be\label{fast}
\beta>N(N-1).
\ee
In particular the conclusions of Theorem \ref{optimalproba} apply in this case. 
\end{coro}

\begin{proof} 
To help the reader, we begin with a first step assuming that the supremum in the definition of $M_{N-1}\Big(\frac{N-1}{N} V\Big)$ is reached by a system of $N-1$ points 
 $x_1,x_2,\dots x_{N-1}$ in $\R^d$ that is
$$M_{N-1}\Big(\frac{N-1}{N} V\Big)=\frac1{N}\sum_{i=1}^{N-1} V(x_i)-c_{N-1}(x_1,\dots,x_{N-1}).$$
Then for every $x_N$, we have:
\be \label{ineqN} M_N(V)\ge M_{N-1}\Big(\frac{N-1}{N} V\Big)+\frac{V(x_N)}{N}-\sum_1^{N-1}\frac1{|x_N-x_i|}.\ee
Now as \eqref{fast} holds, we can choose $|x_N|$ so large to have 
\be \label{conditionN}\frac{V(x_N)}{N}-\sum_{i=1}^{N-1}\frac1{|x_N-x_i|}>0\ee
and \eqref{strictMN} follows.
This proof can be extended to the case where the $N-1$ points infimum related to $M_{N-1}\Big(\frac{N-1}{N} V\Big)$   is not attained in $(\R^d)^{N-1}$, by considering instead of $N-1$ the index $\bar k:=k_{N-1}(V)$ defined by \eqref{def:kbar}. It satisfies:
$$ M_{N-1}\Big(\frac{N-1}{N} V\Big)= M_{N-2}\Big(\frac{N-2}{N} V\Big)=\dots = M_{\bar k}\Big(\frac{\bar k}{N} V\Big)
> M_{\bar k-1}\Big(\frac{\bar k-1}{N} V\Big)\ $$
(notice that $\bar k \ge 1$ since $\beta>0$ implies that $V\not= 0$).
By applying the last statement of Theorem \ref{optimalproba} with $N=\bar k$, we deduce the existence of 
a system of $\bar k$ points  $x_1,x_2,\dots x_{\bar k}$ in $\R^d$ where $\bar k\le N-1$ such that
$$M_{\bar k}\Big(\frac{\bar k}{N} V\Big)=\frac1{\bar k}\sum_{i=1}^{\bar k} V(x_i)-c_{\bar k}(x_1,\dots,x_{\bar k}).$$
Accordingly the counterpart of the inequality \eqref{ineqN} is the following
\begin{align} M_N(V)\ \ge\ & M_{\bar k}\Big(\frac{\bar k}{N} V\Big)+
\sum_{\bar k<j\le N}\frac{V(x_j)}{N} \nonumber\\
&-\sum_{1\le i\le {\bar k}<j\le N}\frac1{|x_i-x_j|}
-\sum_{\bar k<j<l\le N} \frac1{|x_j-x_l|}. \label{ineqbark}
\end{align}  %
As \eqref{ineqbark} holds for all $x_j$ with $ j>\bar k$ and for all  $x_l$ with $ l>j$, by sending $|x_l|\to \infty$ and then
$|x_j| \to \infty$ for $\bar k <j \le N-1$, we are led to
$$ M_N(V)- M_{N-1}\Big(\frac{N-1}{N} V\Big)= M_N(V)\  - M_{\bar k}\Big(\frac{\bar k}{N} V\Big) \ge \frac{V(x_N)}{N} 
- \sum_{1\le i\le {\bar k}} \frac1{|x_i-x_N|}\ , $$
which holds for every $x_N\in \R^d$. The conclusion follows by choosing $|x_N|$ so large to have:
$$\frac{V(x_N)}{N} - \sum_{1\le i\le \bar k} \frac1{|x_i-x_N|} >0\ $$
(note that this  condition is weaker than \eqref{conditionN} if $\bar k <N-1$).
\end{proof}

\begin{rema}\label{t*}
As a consequence of Theorem \ref{optimalproba}, we obtain that optimal solutions belong to $\PP$
as far as $V$ is ``large'' enough. More precisely, if the positive part of $V$
does not vanish, then there exists a constant $t^*\ge0$ such that for $Z>t^*$, the potential $Z V$ satisfies \eqref{strictMN}.  Indeed, by applying \eqref{recession}, we derive that
$$\lim_{Z\to+\infty}\frac{M_k\Big(\frac{k}{N} ZV\Big)-M_{k-1}\Big(\frac{ k-1}{N} ZV\Big)}{Z}=\frac1{N}\,\sup V>0,$$

Moreover, if the potential $V$ is strong enough at infinity, we may even have $t^*=0$. Indeed 
by applying  Corollary \ref{slowexists} to a potential $V\in C_0^+$ satisfying $\beta=+\infty$ ({\em confining potential}), we obtain that $\mathcal{S}_{tV}\subset\PP$ hold for all $t>0$.
 Unfortunately we do not know in general if the opposite condition $t<t^*$ implies that $\mathcal{S}_{tV}\cap\PP$ is empty. In Example \ref{nonexistence}, we merely show that the latter set is empty if $V$ has compact support and $t<t_*$ for a suitable $t_*\le t^*$
\end{rema}
\begin{rema}
 The minimization problem \eqref{minrelax} can be also studied for potentials $V$ possibly unbounded. In fact an important case, which is beyond the scope of this paper, occurs when $\lim_{|x|\to \infty}V=-\infty$. In this case the minimum is reached on probability measures and we observe that, extending the definition of $M_N$ to such potentials, we have a relation with the so called systems of points interactions theory confined by an external potential (see \cite{SL,Ssurvey,SP}), since
$$ -M_N(  -N^2 V) = \inf \left\{ \mathcal{H}_N(x_1,x_2,\dots, x_N) : x_i\in \R^d\right\}$$ 
where $\mathcal{H}_N$ is of the form
$$ \mathcal{H}_N(x_1,x_2,\dots, x_N) = \sum_{1\le<i<j} \ell(|x_i-x_j|) + N \sum_{i=1}^N V(x_i).$$
In such a setting, the asymptotic limit as $N\to\infty$ is one of the main point of interest of the mathematical physics community.
\end{rema}
 

 
\begin{exam}\label{nonexistence}({\em Non existence of an optimal probability})\ 
Let $V\in C_0^+$ with compact support and let $R>0$ such that $\spt V\subset\overline{B_R}$. Then it easy to check that any solution $\rho$ to \eqref{mineps} such that $\|\rho\|=1$ must satisfy as well $\spt\rho\subset\overline{B_R}$. Indeed, otherwise we can move away the part of such a $\rho$ which lies outside $\overline{B_R}$ letting $\int V\,d\rho$ invariant and making $C(\rho)$ decrease. For instance we may consider $\rho'=T^\#(\rho)$ where 
$$T(x)=k_R(|x|)\,x \quad \text{with}\quad k_R(x):=\max\Big\{1,\frac{|x|}{R}\Big\}.$$
For such a map we have $|Tx-Ty|\ge|x-y|$ with strict inequality whenever $(x,y)\notin B_R^2$, so that $C(\rho')<C(\rho)$ unless $\spt\rho\subset\overline{B_R}$.

Next by applying a lower bound estimate for $C(\rho)$ in term of the variance of $\rho$ (see Prop 2.2 in \cite{bbcd18}), we infer that:
$$C(\rho)\ge\frac{N(N-1)}{4\sqrt{{\rm Var}(\rho)}}\ge\frac{N(N-1)}{4R}.$$
On the other hand, by \eqref{miniesti} and as $\rho $ is optimal, it holds $C(\rho) \le (1-\frac1{N}) \sup V$.
 Therefore a solution in $\PP$ cannot exist unless the following lowerbound holds for $\sup V$:
 $$\frac{N^2}{4R} \ \le \ \sup V. $$
This necessary condition applies of course if we substitute potential $V$ with $tV$ and we deduce that $\mathcal{S}_{tV} \cap \PP $ is empty whenever $t<t_*$ where
$$t_*:=\frac{N^2}{4R\,\sup V}.$$
\end{exam}

\bigskip
We are now in a position to state the quantization phenomenon we have announced in the Introduction.
 For every non-vanishing $V\in C_0$, we define:
\begin{equation}\label{INV}
\mathcal {I}_N(V):=\min\left\{\|\rho\|\ :\ \rho\in\mathcal{S}_V\right\}.
\end{equation}
The minimum in \eqref{INV} is achieved by the weak* lower semicontinuity of the map $\rho\mapsto\|\rho\|$ and
optimal $\rho$ represent elements with minimal norm in $\mathcal{S}_V$
On the other hand, $\mathcal{I}_N(V)$ depends only on the positive part of $V$ i.e. $\mathcal{I}_N(V)=\mathcal{I}_N(V^+).$

\begin{theo}\label{fractionalmass}
Let $V\in C_0$, $N\in \N^*$ and $k_N(V)$ given by \eqref{def:kbar}. Then
$$\mathcal{I}_N(V)=\frac{k_N(V)}{N}\;.$$
As a consequence, the map $V\in C_0\mapsto \mathcal{I}_N(V)$ ranges into the finite set
$$\left\{\frac {k}{N}\ :\ 0\le k\le N\right\}.$$
\end{theo}

\begin{proof}
First we observe that the result is trivial if $k_N(V)=0$. Indeed in this case, $V^+\equiv 0$ implies that the minimum in \eqref{minrelax} vanishes. The minimal set is then reduced to $\rho=0$ and $\mathcal{I}_N(V)=0$. We may therefore assume that $k_N(V)\ge1$.
To lighten the notations, let us set now $\bar k:= k_N(V)$.
First we show that $\mathcal{I}_N(V)\le \bar k$. By \eqref{def:kbar}, we have that $M_{\bar k}\big(\frac{\bar k}{N}V\big)>M_{k-1}\big(\frac{\bar k-1}{N}V\big)$ while, recalling the monotonicity property \eqref{ineqMk}, it holds $M_{k}(\frac{ k}{N} V)= M_N(V)$ for every $k\ge\bar k$. Let us apply Proposition \ref{optimalproba} taking instead of $C=C_N$ the $\bar k$ multi-marginal energy and choosing $\bar k V/N$ as a potential. Therefore it exists a probability $\rho_{\bar k}$ such that:
$$C_{\bar k}(\rho_{\bar k})-\int \frac{\bar k V}{N}\, d\rho = - M_{\bar k}\Big(\frac{\bar k V}{N}\Big).$$
Then we claim that $\rho:=\frac{\bar k}{N}\rho_{\bar k}$ solves \eqref{minrelax} (i.e. belongs to $\mathcal{S}_V$). Indeed, by \eqref{easy}, we have:
$$\Cbar(\rho)-\int V\,d\rho\le C_{\bar k}(\rho_{\bar k})-\int\frac{\bar k V}{N}\,d\rho_{\bar k}=-M_{\bar k}\Big(\frac{\bar k V}{N}\Big)=-M_N(V).$$
As the mass of $\rho$ is exacly $\bar k/N$, we infer that $\mathcal{I}_N(V)\le\bar k/N$.

\med
Let us prove now the opposite inequality. Let $\rho\in\mathcal{S}_V$ and let  $\{\rho_k\}$ be an optimal decomposition for $\rho$ according to \eqref{rhok}, with
$$\rho=\sum_{k=1}^N \frac{k}{N}\rho_k.$$
By the optimality conditions of Theorem \ref{prop:cns}, it holds $M_k\big(\frac{k}{N} V\big)=M_N(V)$ whenever $\|\rho_k\|>0$. Then we observe that the latter equality cannot hold for $k\le {\bar k}-1$. Indeed, by the monotonicity property \eqref{ineqMk}:
$$M_k\Big(\frac{k}{N} V\Big)\le M_{\bar k-1}\Big(\frac{\bar k-1}{N} V\Big)<M_{\bar k}\Big(\frac{\bar k}{N} V\Big)=M_N(V).$$
Therefore we have $\rho_k=0$ for every $k\le{\bar k}-1$. Thus recalling that $\sum_k\|\rho_k\|=1$ by the optimality conditions (assertion i) of Theorem \ref{prop:cns}):
$$\|\rho\|=\sum_{k=\bar k}^N \frac{k}{N}\|\rho_k\|\ge\frac{\bar k}{N}\sum_{k=\bar k}^N \|\rho_k\|\ge\frac{\bar k}{N}.$$
Accordingly we obtain the opposite inequality $\mathcal{I}_N(V)\ge\bar k/N$.
\end{proof}

\begin{rema} The functional $V\in C_0 \mapsto \mathcal{I}_N(V)$ is lower semicontinous with respect to the uniform convergence. Indeed if $V_n\to V$ uniformly and if we take $\rho_n \in \mathcal{S}_{V_n}$ such that $\|\rho_n\|=\mathcal{I}_N(V_n)$, then any weak* limit $\rho$ of a subsequence of $(\rho_n)$ is such that $\rho \in \mathcal{S}_{V}$
and 
$$\mathcal{I}_N(V)\le \|\rho\|\le  \liminf_n \|\rho_n\| = \liminf_n \mathcal{I}_N(V_n).$$
\end{rema}

\bigskip
In general the potential $V$ depends on several charge parameters and takes the form
$$ V(x) = \sum_{k=1}^M Z_k \, V_k(x)\, \quad Z_k>0.$$
It is then interesting to analyse the function
$$\mathcal{I}_N(V)=\mathcal{I}_N\Big(\sum_{k=1}^M Z_k\,V_k\Big)$$
as a function depending on the $Z_k$'s.
It turns out that this question is a very delicate one which will  motivate  future works.
In case of a single charge parameter $Z>0$ applied to a given potential $V\in C_0^+$, it is natural to expect that the map $Z\in\R^+\mapsto \mathcal{I}_N(ZV)$ is a non-decreasing step function  encoded by threshold values $0=t_0\le t_1\le\dots\le t_N<t_{N+1}=+\infty$ such that $\mathcal{I}_N(ZV)=\frac{k}{N}$ for $Z\in(t_k,t_{k+1}]$.
At present a proof of this fact is available  only in the case $N=2$ where it is a consequence of the fact that
the set $\{ Z\ge 0: M_2(ZV/2) > M_1(ZV)= Z \sup V\}$ is an half line.

\bigskip

\ack
This paper has been written during some visits of the authors at the Departments of Mathematics of Universities of Firenze and of Pisa, and at the Laboratoire IMATH of University of Toulon. The last part of the work was completed during the visiting fellowship of the first author to CRM-CNRS Laboratorio Fibonacci in Pisa. The authors gratefully acknowledge the warm hospitality and support of these institutions. The work of the second author is part of the PRIN project {\it Gradient flows, Optimal Transport and Metric Measure Structures} (2017TEXA3H), funded by the Italian Ministry of Education and Research (MIUR). The work of the fourth author is partially financed by the {\it``Fondi di ricerca di ateneo''} of the University of Firenze. The second and fourth authors are members of the research group GNAMPA of INdAM.


\bigskip
{\small\noindent
Guy Bouchitt\'e:
Laboratoire IMATH, Universit\'e de Toulon\\
BP 20132, 83957 La Garde Cedex - FRANCE\\
{\tt bouchitte@univ-tln.fr}\\
{\tt https://sites.google.com/site/gbouchitte/home}

\bigskip\noindent
Giuseppe Buttazzo:
Dipartimento di Matematica, Universit\`a di Pisa\\
Largo B. Pontecorvo 5, 56127 Pisa - ITALY\\
{\tt giuseppe.buttazzo@unipi.it}\\
{\tt http://www.dm.unipi.it/pages/buttazzo/}

\bigskip\noindent
Thierry Champion:
Laboratoire IMATH, Universit\'e de Toulon\\
BP 20132, 83957 La Garde Cedex - FRANCE\\
{\tt champion@univ-tln.fr}\\
{\tt http://champion.univ-tln.fr}

\bigskip\noindent
Luigi De Pascale:
Dipartimento di Matematica e Informatica, Universit\`a di Firenze\\
Viale Morgagni 67/a, 50134 Firenze - ITALY\\
{\tt luigi.depascale@unifi.it}\\
{\tt http://web.math.unifi.it/users/depascal/}

\end{document}